\documentclass[10pt]{amsart} 
\usepackage{graphicx}
\usepackage{indentfirst,csquotes}

\usepackage{amssymb,amsthm,amsmath,mathtools}
\usepackage{xcolor,paralist,hyperref,fancyhdr,etoolbox}
\usepackage{scrextend}

\usepackage{bm} 
\usepackage{amsfonts} 
\usepackage{lipsum}

\usepackage[draft]{changes}

\newtheorem{theorem}{Theorem}[section]
\newtheorem{definition}[theorem]{Definition}
\newtheorem{example}[theorem]{Example}
\newtheorem{lemma}[theorem]{Lemma}
\newtheorem{proposition}[theorem]{Proposition}
\newtheorem{corollary}[theorem]{Corollary}

\newtheorem{assumption}[theorem]{Assumption}
\newtheorem{remark}[theorem]{Remark}

\newcommand{\NN}{\mathbb{N}}
\newcommand{\RR}{\mathbb{R}}
\newcommand{\EE}{\mathbb{E}}
\newcommand{\LL}{\mathbb{L}}
\newcommand{\FF}{\mathbb{F}}
\newcommand{\PP}{\mathbb{P}}
\newcommand{\HH}{\mathbb{H}}

\newcommand{\QQ}{\mathbb{Q}}
\newcommand{\barPP}{\bar{\PP}}
\renewcommand{\SS}{\mathbb{S}} 

\newcommand{\calC}{\mathcal{C}}
\newcommand{\calF}{\mathcal{F}}
\newcommand{\calP}{\mathcal{P}}
\newcommand{\calL}{\mathcal{L}}
\newcommand{\calM}{\mathcal{M}}

\newcommand{\calE}{\mathcal{E}}
\newcommand{\calS}{\mathcal{S}}
\newcommand{\calW}{\mathcal{W}}
\newcommand{\calQ}{\mathcal{Q}}
\newcommand{\calK}{\mathcal{K}}
\newcommand{\calY}{\mathcal{Y}}
\newcommand{\calU}{\mathcal{U}}
\newcommand{\frakA}{\mathfrak{A}}
\newcommand{\frakK}{\mathfrak{K}}
\newcommand{\dif}{\mathrm{d}}
\newcommand{\abs}[1]{\left|#1\right|}
\newcommand{\norm}[1]{\left\|#1\right\|}

\newcommand{\boldsym}[1]{\boldsymbol{#1}}
\newcommand{\boldnu}{\boldsym{\nu}}
\newcommand{\barcalL}{\bar{\calL}}
\newcommand{\hatalph}{\hat{\alpha}}
\newcommand{\hatPP}{\hat{\PP}}
\newcommand{\hatcalL}{\hat{\calL}}

\usepackage[
backend=biber,
style=numeric,
date=year,
sorting=nty,
giveninits=true,
doi=false,
isbn=false,
url=false,
eprint=false]{biblatex}
\hypersetup{ colorlinks=true, linkcolor=black, filecolor=black, urlcolor=black }
\renewenvironment{proof}[1][\proofname]{\noindent{\it #1. $\,$}}{\hfill$\Box$\vskip 5pt}
\def\proofname{Proof}

\numberwithin{equation}{section}

\begin{document}

\title[Global Solutions of Non-Markovian MFGC]{Mean-Field Games with Unbounded Controls:\\ A Weak Formulation Approach to Global Solutions}\thanks{Horst gratefully acknowledges financial support from DFG CRC/TRR 388 ``Rough
Analysis, Stochastic Dynamics and Related Fields''- Project ID 516748464. Sato gratefully acknowledges financial support with a scholarship from the German Academic Exchange Service (DAAD). We thank Huilin Zhang for many valuable comments and discussions.}

\author{Ulrich Horst}
 \address{Department of Mathematics and School of Business and Economics, Humboldt-Universit\"at zu Berlin, Unter den Linden 6, 10099 Berlin, Germany}
 \email{horst@math.hu-berlin.de}

\author{Takashi Sato}
 \address{Department of Mathematics, Humboldt-Universit\"at zu Berlin, Unter den Linden 6, 10099 Berlin, Germany}
\email{sato.takashi@hu-berlin.de}


\date{\today}

\maketitle

\begin{abstract}
We establish an existence of equilibrium result for a class of non-Markovian mean-field games with unbounded control space in weak formulation. Our result is based on  new existence and stability results for quadratic-growth generalized McKean-Vlasov BSDEs. Unlike earlier approaches, our approach does not require boundedness assumptions on the model parameters or time horizons and allows for running costs that are quadratic in the control variable.   
\end{abstract}

\medskip



\renewcommand{\baselinestretch}{1.025}

\section{Introduction}

Mean field games (MFGs) provide a powerful mathematical framework for analyzing strategic interactions in large populations of agents. Introduced independently by Lasry and Lions \cite{Lasry2006,Lasry2007} and by Huang, Caines, and Malhamé \cite{Huang2007,huang2006}, MFGs describe the asymptotic behavior of Nash equilibria in stochastic differential games as the number of players tends to infinity. In this limit, each individual agent has a negligible influence on the overall system but interacts with the population through the empirical distribution of states and/or controls.

Formally, for a given measure-flow $\mu$, a representative agent in an MFG selects an admissible control $\alpha$ to maximize a cost functional of the form
\[
    J(\alpha) := \EE\left[\int_0^Tf(t,X_t,\mu_t,\alpha_t)\,dt+g(X_T,\mu_T)\right]
\]
subject to the state dynamics 
\[
    dX_t =b(t,X_t,{\mu}_t,\alpha_t)\,dt+\sigma(t,X_t,{\mu}_t,\alpha_t)\,dW_t, \quad X_{0} =x_0,
\]
and the condition that $\mu$ coincides with the law of the optimally controlled state process.

Several approaches have been developed to solve such problems. In their original formulation \cite{Lasry2006,Lasry2007}, Lasry and Lions adopted an analytic perspective and characterized equilibria through a coupled system of nonlinear partial differential equations: a backward Hamilton–Jacobi–Bellman (HJB) equation for the representative agent’s value function and a forward Fokker–Planck (Kolmogorov) equation describing the evolution of the population distribution.

The probabilistic approach to MFG theory was pioneered by Carmona and Delarue in \cite{carmona2013probabilistic}. Using a Pontryagin-type stochastic maximum principle, they showed that MFG equilibria can be characterized in terms of a McKean–Vlasov forward–backward SDE (MV-FBSDE).

A relaxed solution concept for MFGs was introduced by Lacker in \cite{lacker2015mean} and later extended to MFGs with common noise by Carmona et al.~\cite{carmona2016mean}, and to MFGs with singular controls by Fu and Horst \cite{fu_horst_2017}. In the spirit of traditional game theory the key idea when working with relaxed controls is to first establish upper hemicontinuity of the agent’s best-response correspondence with respect to a given measure-flow, and then to apply a fixed-point argument. 

Since their introduction, MFGs have found numerous applications in engineering, finance, economics, and operations research. These range from decentralized control of robotic swarms, communication networks, and automated vehicles \cite{Achdou2010,Bensoussan2013},
to portfolio choice problems under performance concerns \cite{Carmona2015,FuZhou2023MFGPortfolio,Lacker2017Portfolio}, to optimal trade execution under market impact \cite{Cardaliaguet2018,FHH-2023,FHX1,FGHP-2018}, and optimal exploitation of exhaustible resources \cite{CS-2017,PUR}.

\subsection{Our contribution}

We consider a class of non-Markovian mean-field games of control in weak formulation with possibly non-compact action sets. The weak formulation, first introduced by Carmona and Lacker \cite{carmona_lacker_2015} and further developed in \cite{carmona_wang_2021,lacker2015mean,possamai_talbi_2023,possamai2025non}, avoids the direct use of MV-FBSDEs. In this setting, the FBSDE system arising in the probabilistic formulation reduces to a generalized McKean–Vlasov BSDE of the form
\begin{equation}\label{intro-MVBSDE}
\left\{ 
\begin{split} 
\dif X_t &= \sigma_t(X) \dif W_t, \quad X_0 = x_0, \\
\dif Y_t &= -H_t(X, Z_t, \barcalL(X,Z_t)) \dif t + Z_t \dif W_t, \quad Y_T = G(X,\barcalL(X)), \\
\frac{\dif \barPP}{\dif \PP} &= \calE\left( B_\cdot (X, Z_\cdot, \barcalL(X)) \cdot W \right)_T, \quad \barcalL (\cdot) := \barPP \circ (\cdot)^{-1}. 
\end{split} \right. 
\end{equation}

We establish novel existence results for equations of this type under quadratic growth conditions on the driver. While MV-(F)BSDEs with quadratic drivers in a strong formulation have been extensively studied in the literature \cite{buckdahn2026global,hao2025mean,hu2022quadratic,jiang2024general}, 
to the best of our knowledge this is the first work to consider generalized MV-BSDEs of the form \eqref{intro-MVBSDE} with quadratic drivers.

Our results enable us to prove the existence of equilibria in MFGs with path-dependent and possibly discontinuous (in the state variable) running cost functions that may exhibit quadratic growth in the control variable. Wang and Zhang \cite{wang2018weakFBSDE} showed that optimal controls, and hence MFG equilibria, may fail to exist in strong formulation when the running cost is path-dependent and discontinuous in the state variable.

Compactness of the control space is frequently assumed in the literature, but this excludes games with purely quadratic running costs. Lacker \cite{lacker2015mean} is among the few works that relax the compactness assumption on the control space; however, mean-field interaction through the control variable is not permitted in his framework. Our setting is also more flexible than that of Carmona and Lacker \cite{carmona_lacker_2015}, as we do not require the drift $B$ to be bounded. This allows, for example, for a geometric Brownian motion with an unbounded controlled drift as the state dynamics. These generalizations come at the expense of additional boundedness assumptions on the state variable, as our analysis relies heavily on existence results for quadratic BSDEs.

Our work is strongly inspired by Possamaï and Tangpi \cite{possamai2025non}. To our knowledge, they were the first to characterize MFG equilibria in weak formulation via generalized MV-BSDEs of the form \eqref{intro-MVBSDE}. They established existence and uniqueness of equilibria, together with non-asymptotic convergence rates for finite-player games, under boundedness assumptions on the action sets and suitable Lipschitz continuity conditions, as well as differentiability assumptions on the model parameters or smallness conditions on the terminal payoff. In contrast, we do not consider finite-player games and instead focus on the existence of equilibria in MFGs with possibly non-compact action sets and quadratic drivers.

Our approach consists in replacing the law arguments in the generalized MV-BSDE with generic parameters and defining a solution mapping that associates these parameters with the laws of the corresponding BSDE solutions. Existence then reduces to proving continuity of this mapping and identifying a compact, convex subset that is mapped into itself. A key novelty of our approach lies in the use of uniform boundedness of the $Z$-components of solutions to quadratic BSDEs in the BMO norm, without requiring smallness conditions on the model parameters or compactness of the control space. This distinguishes our work from existing approaches. In \cite{carmona_lacker_2015}, the solution mapping is set-valued and involves maximized Hamiltonians, and equilibrium existence is obtained via a Measurable Maximum Theorem \cite[Theorem 18.19]{guide2006infinite}, which typically requires compact action spaces. In \cite{possamai2025non}, this issue is avoided by fixing a measurable maximizer of the Hamiltonian in advance and incorporating it into their MV-BSDE. Regardless of the specific approach, in both papers compactness of the action space is required to guarantee Lipschitz continuity of BSDE drivers as both papers rely on stability results for BSDEs with Lipschitz drivers. We use stability results for quadratic BSDEs.

Since the $Z$-component of a BSDE solution need not be continuous, compactness in the space of probability measures is difficult to establish directly. To address this difficulty, we lift the solution mapping to the space of integrable Young measures. The use of Young measures in the context of MV-BSDEs was pioneered in \cite{carmona_lacker_2015}; we extend this methodology to  quadratic settings, which raises several technical challenges.\footnote{Since the space of all Young measures is too large in our setting, due to the relaxed boundedness assumptions on the mean-field terms, we actually need to work with integrable Young measures.} First, continuity of the solution mapping with respect to the stable topology on integrable Young measures must be established. We address this by proving a new stability result for quadratic BSDEs, which is of independent interest as a stability result with respect to the driver. Second, one must identify a compact set of measures that is invariant under the solution mapping. When the model parameters are bounded in the mean-field terms, such a set can be characterized using a uniform BMO bound for solutions of quadratic BSDEs. When the model parameters are unbounded in the mean-field terms, more delicate arguments are required. In this case, uniform BMO bounds are not available to the best of our knowledge. However, by exploiting the equivalence between generalized MV-BSDEs and MV-FBSDEs, and by adapting an a priori estimate due to Hao et al.~\cite{hao2025mean}, we derive an upper bound for the BMO norm of the $Z$-component of any solution to our generalized MV-BSDE, which allows us to establish existence via a truncation argument.

\subsection{Organization of the paper}

Section \ref{sec:setup} introduces the model, key assumptions, and main results. Rather than stating all assumptions at once, we deliberately develop the theory progressively (as there are many subtleties), 
beginning with a general characterization of MFG equilibria in terms of a generalized MV-BSDE. We then specify conditions on the model parameters and admissible control sets under which the BSDE admits a comparison principle. Existence of solutions is established under standard separability conditions on the running cost, and—when model parameters are unbounded—under an additional strong quadratic growth condition on the driver. Section \ref{sec:stability} proves continuity of the solution mapping in the stable topology. Section \ref{sec:cpt} establishes a compactness result for integrable Young measures and identifies a compact, convex invariant set. Section \ref{sec:existence} proves the existence of equilibria in weak formulation.

\subsection{Notation.} 
We set $\NN^\ast := \NN \cup \{\infty\}$ and write $\calP(E)$ for the set of probability measures on a measurable space $(E, \calE)$. Expectations w.r.t.~a measure $\mu$ are denoted  $\EE^\mu$ or simply $\EE$. The $p$-th moment of $\mu$ is denoted $M_p(\mu)$, and $\calP_p(E) := \{\mu \in \calP(E) : M_p(\mu) < \infty\}.$ For two measures $\mu, \mu' \in \calP_p(E)$, we denote by $\calW_p(\mu,\mu')$ their Wasserstein distance induced by a metric on $E$.
Unless otherwise specified, $\calP_p(E)$ is endowed with the $\calW_p$-topology. 
$\LL^p (\RR^l;\QQ)$ denotes the usual space of $\RR^l$-valued measurable functions with finite $L^p$-norm under $\QQ$ and $\HH^p(\RR^l;\QQ)$ denotes the space of $\RR^l$-valued predictable processes $H$ on $[0,T]$ with $\| (\int_0^T \abs{H_s}^2 \dif s )^{\frac12}\|_{\LL^p(\RR;\QQ)} < \infty$.  
For progressively measurable processes $H$ we set $\norm{H}_{\SS^p(\RR^l;\QQ)} := \|  \sup_{t \in [0,T]} \abs{H_t}\|_{\LL^p(\RR; \PP)}$ and $\norm{H}_{\HH_{\operatorname{BMO}}^2(\RR^l;\QQ)}$ denotes the usual BMO-norm.
$\SS^p(\RR^l;\QQ)$ denotes the space of $\RR^l$-valued adapted continuous processes on $[0,T]$ with finite $\norm{\cdot}_{\SS^p(\RR^l;\QQ)}$-norm, and $\HH^2_{\operatorname{BMO}}(\RR^l;\QQ)$ is the space of $\RR^l$-valued predictable processes on $[0,T]$ with finite $\norm{\cdot}_{\HH^2_{\operatorname{BMO}}(\RR^l;\QQ)}$-norm.
When no confusion arises, we write $L^p(\QQ), \HH^p(\QQ), \SS^p(\QQ)$, and $\HH^2_{\operatorname{BMO}}(\QQ)$ for short.

\section{Setup and main results}\label{sec:setup}

Throughout this paper, all randomness is described by a probability space $(\Omega, \calF, \PP)$ that carries a $d$-dimensional Brownian $W$. We denote by $\FF = \{\calF_t\}$ the natural filtration generated by $W$ augmented by the $\PP$-null sets. Stochastic integrals of progressively measurable process $H$ w.r.t.~$W$ are denoted $H \cdot W$. We fix a time horizon $T<\infty$, denote by $\calC_d$ the set of $\RR^d$-valued continuous functions on $[0,T]$ equipped with the supremum norm $\norm{\cdot}_\infty$.

\subsection{Mean-field game in weak formulation}

To introduce our non-Markovian MFG in weak formulation, we fix measurable diffusion coefficients\footnote{We implicitly assume that the space $\calC_d \times \RR^d$ is endowed with the norm $\norm{(\cdot, \cdot)} := \norm{\cdot}_\infty + \abs{\cdot}$.}
\[
	b:[0,T] \times \calC_d \times \calP (\calC_d \times A) \times A \to \RR^d,  \quad
	\sigma:[0,T] \times \calC_d \to \RR^{d \times d}.
\]
For $x \in \calC_d$, we set $[x]_t :=(x_s)_{0 \leq s \leq t}$ and assume throughout that 
\[
    b_t(x,q,a) = b_t ({[x]_t, q \circ \pi_t^{-1},a), \quad
    \sigma_t(x) = \sigma_t([x]_t}), \quad
    \text{where} \quad \pi_t(x,a) := ([x]_t,a).
\]

For an initial state $x_0 \in \RR$, we consider the $\RR^d$-valued state process $X$ on $[0,T]$ with dynamics
\begin{equation}\label{driftless-SDE}
    \dif X_t = \sigma_t(X) \dif W_t, \quad X_0 = x_0.
\end{equation}
We assume that the volatility process satisfies the following conditions under which our SDE admits a unique strong solution $X \in \bigcap_{p < \infty} \SS^p(\PP)$ for any initial state.

\begin{assumption}\label{asm:driftless-SDE}
The volatility process $\sigma$ satisfies the following conditions:
\begin{enumerate}
	\item\label{cond:L-cont-sigma}
    $\sigma_\cdot (0) \in L^2([0,T];\RR^d)$ and $\abs{\sigma_t(x) - \sigma_t(\bar{x})} \leq K_\sigma \norm{x - \bar{x}}_\infty$
	\item
    the $d \times d$-diffusion matrix $\sigma_t(X)$ is invertible $\PP$-almost surely.
\end{enumerate}
\end{assumption}

In what follows, we denote by $\calM_1$ the class of Borel measurable measure-flows 
\[
    m:[0,T] \to \calP_1(\calC_d \times \RR^d) \quad \mbox{with} \quad  \int_0^T M_1(m_t) \dif t < \infty.
\]
We frequently write $m = (m^x, m^a)$ to indicate the marginal distributions on the path-space and control set. We fix a (possibly unbounded) set of admissible actions $A \subset \RR^k$ and denote by $\frakA$ a class of predictable control processes $\alpha: [0,T] \times \Omega \to A$ that may take values in an unbounded set. The precise definition of admissibility is given below. For now, we only assume that for any $m \in \calM_1$ and $\alpha \in \frakA$ the measure $\PP^{\alpha, m}$ with density
\begin{equation}\label{density-P-alpha-m}
	\frac{\dif \PP^{\alpha,m}}{\dif \PP} = \calE\left((\sigma^{-1}b)_\cdot \big(X, m_\cdot, \alpha_\cdot \big) \cdot W \right)_T 
\end{equation}
on $\calF_T$ is well defined, where $\calE(M) := \exp\{M - \frac12 \langle M \rangle \}$ for a continuous local martingale $M$.

\begin{remark}\label{rmk:Girsanov}
Under the measure  $\PP^{\alpha, m}$, the process 
\[
	W_t^{\alpha, m} := W_t - \int_0^t (\sigma^{-1}b)_s (X, m_s, \alpha_s) \dif s
\]
is a Brownian motion and the state process $X$ satisfies the SDE 
\[
	\dif X_t = b_t (X, m_t, \alpha_t) \dif t + \sigma_t (X)  \dif W_t^{\alpha, m}. 
\]
This justifies working with martingale state processes.
\end{remark}

To define the objective function, we fix the running and terminal payoff functions 
\[
	f:[0,T] \times \calC_d \times \calP (\calC_d \times A) \times A \to \RR, \quad \mbox{and} \quad
	g:\calC_d \times \calP (\calC_d) \to \RR.
\]
Analogously to the diffusion coefficients, we assume that 
\[
    f_t(x,q,a) = f_t([x]_t, q \circ \pi_t^{-1}, a).
\]

For any measure-flow $m = (m^x, m^a) \in \calM_1$ the representative player chooses a control to maximize the expected payoff   
\begin{equation}\label{cost-function}
	J_W (\alpha, m)
	:=  \EE^{\PP^{\alpha,m}} \left[\int_0^T f_s (X, m_s, \alpha_s) \dif s + g(X, m_T^x) \right],
\end{equation}
over the class of admissible controls $\frakA$. This implicitly assumes that the payoff function $J_W (\alpha, m)$ is well defined. We hence employ the following definition of admissibility. 

\begin{definition}
An $\RR^d$-valued, $\FF$-predictable control $\alpha$ is called admissible if the measure change \eqref{density-P-alpha-m} and the payoff functional \eqref{cost-function} are well defined for all $m \in \calM_1$. 
\end{definition}

 Our goal is to establish the existence of an equilibrium in our mean-field game in weak formulation in the sense of the following definition. 

\begin{definition}[Mean-field equilibrium]\label{def:MFE}
Let $\bar \frakA \subseteq \frakA$ be a subset of admissible controls. A control $\hatalph \in \bar \frakA$ is a mean-field equilibrium in weak formulation if there exists a measure-flow $\hat{m} \in \calM_1$ such that
\begin{gather*}
	J_W (\hatalph, \hat{m}) \geq J_W (\alpha, \hat{m}) \quad \mbox{for all} \quad \alpha \in \bar \frakA 
\end{gather*}
and the following fixed-point condition holds:	
\[
	\hat{m}_t = \PP^{\hatalph, \hat{m}} \circ (X, \hatalph_t)^{-1} \quad \text{for a.e. } t \in [0,T].
\]
\end{definition}


\subsection{Generalized McKean-Vlasov BSDEs and MFG Equilibria}\label{sec:MFG-to-FBSDE}

In this section, we derive a characterization of equilibria in weak formulation in terms of a generalized MV-BSDE. This section is strongly inspired by the work of Possamaï and Tangpi \cite{possamai2025non}. However, there are important (and sometimes subtle) differences. Since our model parameters may be unbounded or exhibit quadratic growth, the set of admissible controls is different and so is the proof of the existence of equilibria. 
The verification of the comparison principle is also different as we work with {\sl stochastic} Lipschitz drivers.


\subsubsection{A sufficient maximum principle for MFGs}

We first introduce the following abstract result that establishes a link between equilibria in weak formulation and solutions to generalized MV-BSDEs. It can be viewed as a sufficient maximum principle for MFGs in weak formulation. 

\begin{theorem}(\cite[Proposition 2.8]{possamai2025non})\label{thm:FBSDE-to-MFE}
Let $(t,x,z,m) \mapsto \Lambda_t(x,z,m)$ be a measurable maximizer of the (reduced) Hamiltonian 
\begin{equation}\label{Hamiltonian}
h_t (x,z,m,a) := f_t (x,m,a) + (\sigma^{-1}b)_t (x,m,a) \cdot z.
\end{equation}
Let $\hatalph \in \bar \frakA$ be an admissible control, let $\hat \PP \in \calP(\Omega)$ be a probability measure that is absolutely continuous w.r.t.~$\PP$, and let
\[
	(X,Y,Z) \in \SS^2(\PP) \times \SS^2(\PP) \times \HH^2(\PP)
\]
be a triple of processes that satisfy the (generalized) MV-BSDE 
\begin{equation}\label{MFE-FBSDE}
\begin{split}
\dif X_t &= \sigma_t(X) \dif W_t, \quad X_0 = x_0, \\
\dif Y_t &= -h_t(X, Z_t, \hatcalL(X, \hatalph_t), \hatalph_t) \dif t + Z_t \dif W_t, \quad Y_T = g(X,\hatcalL(X)),
\end{split}
\end{equation}
where $\hatcalL(\cdot) := \hatPP \circ (\cdot)^{-1}$ denotes the law of a random variable under $\hatPP$. If $\hatcalL(X, \hatalph_\cdot) \in \calM_1$ and $\hatPP$ is absolutely continuous with density 
\begin{equation} \label{density-hatP}
    \frac{\dif \hatPP}{\dif \PP}
    = \calE \left((\sigma^{-1}b)_\cdot \big(X, \hatcalL(X, \hatalph_\cdot), \hatalph_\cdot \big) \cdot W\right)_T,
\end{equation}
on $\calF_T$ and if the control $\hatalph$ satisfies the fixed-point condition
\begin{equation}\label{fixed-point-alpha}
	\hatalph_t = \Lambda_t (X, Z_t, \hatcalL(X, \hatalph_t))
\end{equation}
and if the BSDEs with drivers
\[
	 H^\alpha_t(z) := h_t (X, z, \hatcalL(X, \hatalph_t), \alpha_t) \quad (\alpha \in \bar \frakA) 
\]
and terminal condition
\[
    g(X,\hatcalL(X))
\]    
admit unique solutions and satisfy a comparison principle, then $\hatalph$ is an equilibrium with payoff 
\[
    Y_0 = J_W (\hatalph, \hatcalL(X, \hatalph_\cdot)).
\]
\end{theorem}



\subsubsection{Comparison}\label{sec:comparison}

The driver $H^\alpha$ is Lipschitz continuous with the stochastic constant
\begin{equation} \label{K}
	 K_\cdot := (\sigma^{-1}b)_\cdot(X, \hat{m}_\cdot, \alpha_\cdot).
\end{equation}
Well-posedness of BSDEs with stochastic Lipschitz constants has been established by Briand and Confortola \cite{briand2008bsdes} given that $K \in \HH^2_{\operatorname{BMO}}(\PP)$.
This motivates the following assumption.

\begin{assumption}\label{asm:comparison}
For any measure-flow $m \in \calM_1$ and any admissible control $\alpha \in \frakA$ we have
\[
	(\sigma^{-1}b)_\cdot (X, m_\cdot, \alpha_\cdot) \in \HH^2_{\operatorname{BMO}}(\PP).
\]
Furthermore, 
\[
	g(X,m_T^x) + \int_0^T \abs{f_s(X, m_s, \alpha_s)} \dif s \in \LL^p(\PP) \quad \mbox{for any $p \in [1,\infty)$}.
\]
\end{assumption}

The next result follows from Proposition~\ref{prop:comparison-thm} under Assumption~\ref{asm:comparison}.

\begin{proposition}\label{prop:stocLip-BSDE-alpha}
Under the above assumption, for any $m \in \calM_1$ the BSDEs 
\begin{equation}\label{BSDE-alpha}
\begin{split}
    Y_t^\alpha 
    &= g(X, m_T^x) + \int_t^T  h_s (X, Z^\alpha_s, m_s, \alpha_s)  \dif s - \int_t^T Z_s^\alpha \dif W_s
\end{split}    
\end{equation}
defined in terms of the admissible controls $\alpha \in \frak{A}$ admit unique solutions 
\[
	(Y^\alpha, Z^\alpha) \in \bigcap_{p < \infty} (\SS^p(\PP) \times \HH^p(\PP)) 
\]	
and satisfy a comparison principle.
\end{proposition}


\subsubsection{Admissibility} It turns out that Assumption \ref{asm:comparison} is satisfied if we impose a suitable growth condition on the model parameters and take $\frakA := \HH^2_{\operatorname{BMO}}(\PP)$ as our set of admissible controls. 

\begin{assumption}\label{asm:MFG-growth-conditions}
The functions $\sigma^{-1}b$, $f$, and $\Lambda$ satisfy the following conditions for any $(t, a, z,m) \in [0,T] \times A \times \RR^d \times \calP_1(\calC_d \times A)$.
\begin{enumerate}
    \item \label{item:Lp-g-f}
    For any $p \in [1,\infty)$,
    \begin{equation*}
        g(X, \delta_0) + \int_0^T \abs{f_s (X, \delta_0, 0)} \dif s \in L^p(\PP).
    \end{equation*}
    \item \label{item:BMO-sigma-b-Lambda}
    It holds
    \begin{gather*}
        (\sigma^{-1}b)_\cdot (X, \delta_0, 0) \in \HH^2_{\operatorname{BMO}}(\PP), \quad \Lambda_\cdot (X, 0, \delta_0) \in \HH^2_{\operatorname{BMO}}(\PP).
    \end{gather*}
    \item \label{item:growth-game-parameters} 
    There exist constant $\gamma>0$ s.t.~the following growth conditions hold a.s.
    \begin{align*}
        \abs{g(X,m^x)} &\leq \abs{g(X,\delta_0)} + \gamma M_1(m^x), \\
        \abs{(\sigma^{-1}b)_t (X, m, a)}
        &\leq \abs{(\sigma^{-1}b)_t (X, \delta_0, 0)} + \gamma (\abs{a} + M_1 (m)), \\
        \abs{f_t(X, m, a)} 
        &\leq \abs{f_t (X, \delta_0, 0)} + \gamma (\abs{a}^2 + M_1(m)), \\
        \abs{\Lambda_t(X, z, m)}
        &\leq \abs{\Lambda_t (X, 0, \delta_0)} + \gamma (1 + \abs{z}).
    \end{align*}
\end{enumerate}
\end{assumption}

\begin{proposition}\label{prop:admissibility}
Under Assumption \ref{asm:MFG-growth-conditions}, the following holds.
\begin{enumerate}[(i)]
	\item Any control $\alpha \in \HH^2_{\operatorname{BMO}}(\PP)$ is admissible. 
	\item\label{item:BSDE-to-MFE} If there exist an $\FF$-predictable $A$-valued process $\hatalph$ and a solution
	\[
		(X,Y,Z,\hatPP) \in \SS^2(\PP) \times \SS^\infty(\PP) \times \HH^2_{\operatorname{BMO}}(\PP) \times \calP(\Omega) 
	\]	
	to the generalized BSDE \eqref{MFE-FBSDE}-\eqref{fixed-point-alpha} with $\hatcalL(X, \hatalph_\cdot) \in \calM_1$, then $$\hatalph \in \HH^2_{\operatorname{BMO}}(\PP),$$ and the BSDEs \eqref{BSDE-alpha} with $m_\cdot = \hatcalL(X, \hatalph_\cdot)$  admit unique solutions and satisfy a comparison principle. In particular, $\hatalph$ forms a mean-field equilibrium in weak formulation.
\end{enumerate}
\end{proposition}
\begin{proof}
\begin{enumerate} 
\item[(i)] Under Assumption \ref{asm:MFG-growth-conditions} the density $\dif \PP^{\alpha, m} \over \dif \PP$ in \eqref{density-P-alpha-m} is well defined for any $\alpha \in \HH^2_{\operatorname{BMO}}(\PP)$. To see that the cost function $J_W(\alpha, m)$ in \eqref{cost-function} is well defined, it is enough to verify 
\[
\EE^{\alpha, m} \left[ \int_0^T \abs{f_s(X, \delta_0, 0)} \dif s \right]
+ \EE^{\alpha, m} \left[ \int_0^T \abs{\alpha_s}^2 \dif s \right] 
+ \EE^{\alpha, m} \left[ \abs{g(X,m_T)} \right] < \infty.
\]
Proposition~\ref{prop:BMO-norm-equivalence} implies that $\alpha \in \HH^2_{\operatorname{BMO}}(\PP^{\alpha,m})$; hence the second term is integrable.
The reverse H\"older inequality (Proposition~\ref{prop:BMO-chara}) yileds $q > 1$ such that
\[
\EE^{\PP^{\alpha,m}}\left[ \int_0^T \abs{f_s(X, \delta_0, 0)} \dif s \right] 
\leq  ~ \EE\left[ \abs{\frac{\dif \PP^{\alpha, m}}{\dif \PP}}^q \right] \EE \left[ \left( \int_0^T \abs{f_s(X, \delta_0, 0) } \dif s \right)^p \right] < \infty,
\]
where $p>1$ is the conjugate of $q$. Similarly, $\EE^{\PP^{\alpha,m}}[\abs{g(X,m_T)}] < \infty$. 
\item[(ii)] The fact that $\hatalph \in \HH^2_{\operatorname{BMO}}(\PP)$ follows from Assumption~\ref{asm:MFG-growth-conditions} (3).
Since Assumption~\ref{asm:comparison} is satisfied under Assumption~\ref{asm:MFG-growth-conditions}, the assertion follows from Proposition~\ref{prop:stocLip-BSDE-alpha}. The equilibrium property follows from Theorem~\ref{thm:FBSDE-to-MFE}.
\end{enumerate}
\end{proof}


In view of the preceding result,  we assume that the set of admissible controls is given by 
\[
	\frakA :=  \HH^2_{\operatorname{BMO}}(\PP).
\]


\subsubsection{Existence of solutions}
We establish the existence of an equilibrium under the usual separability conditions on the drift and running cost function. 

\begin{assumption} \label{asm:separability}
The drift $b$ is independent of the law of the control. That is,    
\[
    b_t (x, m, a) = b_t (x, m^x, a).
\]
Moreover, the running cost function $f$ satisfies the separability condition 
\[
    f_t (x, m, a) = f_t^1 (x, m^x, a) + f_t^2 (x, m),
\]
for measurable functions $f^1$ and $f^2$.  
\end{assumption}

Under Assumption~\ref{asm:separability}, the fixed-point condition reduces to the purely functional relation  
\begin{equation}\label{fixed-point-alpha1}
	\hatalph_t = \Lambda_t \big( X, Z_t, \hatcalL(X) \big),
\end{equation}
which we inserted into the driver of the BSDE \eqref{MFE-FBSDE}. Moreover, the density \eqref{density-hatP} reduces to
\begin{equation} \label{density1}
    \frac{\dif \hatPP}{\dif \PP}
    = \calE \left((\sigma^{-1}b)_\cdot \big( X, \hatcalL(X), \hatalph_\cdot \big) \cdot W \right)_T.
\end{equation}

In what follows, a generic variable $q = (q^x,q^z) \in \calP_1(\calC_d \times \RR^{d})$ represents the joint law of the process $(X,Z)$ and
\[
    \theta_t^q (x,z) := \big(x, \Lambda_t(x,q^x,z) \big). 
\]
Finding an equilibrium in weak formulation then reduces to first solving the MV-BSDE  
\begin{equation}\label{general-MVBSDE}
\begin{split}
    \dif X_t &= \sigma_t(X) \dif W_t, \quad
    X_0 = x_0, \\
    \dif Y_t &= -H_t(X, Z_t, \barcalL(X,Z_t)) \dif t + Z_t \dif W_t, \quad Y_T = G(X,\barcalL(X)), \\
    \frac{\dif \barPP}{\dif \PP} &= \calE\left( B_\cdot (X, Z_\cdot, \barcalL(X)) \cdot W \right)_T, \quad
    \barcalL (\cdot) := \barPP \circ (\cdot)^{-1},
\end{split}
\end{equation}
with terminal condition $G(x,q^x) := g(x,q^x)$ and driver/maximized Hamiltonian 
\[
	H_t(x,z,q) := F_t(x,z,q) + z \cdot B_t(x,z,q^x) 
\]
where 
\begin{align*}
	F_t(x,z,q) & := f_t \big(x, q \circ (\theta_t^q)^{-1} ,\Lambda_t(x,z,q^x) \big), \quad
    B_t(x,z,q^x) &:= (\sigma^{-1} b)_t \big( x,q^x,\Lambda_t(x,z,q^x) \big)
\end{align*}
and then verifying the integrability condition 
$
    \barcalL(X, \hatalph_\cdot) \in \calM_1.
$

\begin{remark}
    Under Assumptions~\ref{asm:MFG-growth-conditions} and \ref{asm:separability}, the above integrability condition holds if $\barcalL(X,Z_\cdot) \in \calM_1$. In fact, since $\hatalph_t = \Lambda_t \big( X, Z_t, \hatcalL(X) \big)$, Condition~\eqref{item:growth-game-parameters} in Assumption~\ref{asm:MFG-growth-conditions} yields 
   \begin{align*}
    \int_0^T& M_1 \big(\hatcalL(X, \hatalph_s) \big) \dif s \leq C + \EE^{\barPP} \left[ \int_0^T \abs{\Lambda_s(X, 0, \delta_0)} \dif s \right] + \gamma_\Lambda \EE^{\barPP} \left[\int_0^T \abs{Z_s} \dif s \right].
    \end{align*}
    The integrability $\hatcalL(X, \hatalph_\cdot) \in \calM_1$ follows from Condition \eqref{item:BMO-sigma-b-Lambda} in Assumption~\ref{asm:MFG-growth-conditions}.
\end{remark}

 The preceding remark motivates the following definition. 

\begin{definition}\label{def:sol-MVFBSDE}
We call a quadruple $$(X,Y,Z, \barPP) \in \SS^2(\PP) \times \SS^2 (\PP) \times \HH^2_{\operatorname{BMO}}(\PP) \times \calP(\Omega)$$ a solution to the generalized MV-BSDE \eqref{general-MVBSDE} if $\barPP \ll \PP$, the measure-flow $\barcalL(X,Z_\cdot)$ belongs to $\calM_1$, and the process $(X, Y, Z)$ satisfies the generalized MV-BSDE \eqref{general-MVBSDE}. 
\end{definition}

It turns out that our MV-BSDE admits a solution under standard assumptions on model parameters if the parameters are bounded in the mean field term.  

\begin{assumption}\label{asm:BSDE-standing-asm}
The model parameters satisfy the following assumptions for any $(t, x) \in [0,T] \times \calC_d$ and $(z,q), (\bar{z}, \bar{q}) \in \RR^d \times \calP_1(\calC_d \times \RR^d)$. 
    \begin{enumerate}
    \item \label{cond:bdd-in-x} The functions $B$, $G$, and $F$ are bounded in $x$ and of linear growth in $q$; for some $L>0$, 
    \[
    \abs{G(x,q^x)} + \abs{B_t(x,0,q^x)} + \abs{F_t(x,0,q)} \leq L (1 + M_1(q)).
    \]
    \item \label{cond:Lip-continuity} The functions
    $B$ and $F$ are (locally) Lipschitz continuous in $z$; for some $K>0$,
    \begin{align*}
        \abs{B_t(x,z,q^x) - B_t(x,\bar{z},q^x)}
        &\leq K \abs{z-\bar{z}}, \\
        \abs{F_t(x,z,q) - F_t(x,\bar{z},q)}
        &\leq K (1 + \abs{z} + \abs{\bar{z}})(\abs{z-\bar{z}}).
    \end{align*}
    \item The following mappings are continuous: 
    \[
     q^x \mapsto B_t(x,z,q^x), \quad q^x \mapsto G(x,q^x), \quad q \mapsto F_t(x,z,q).
    \]
\end{enumerate}
\end{assumption}

We now state the first main result of this paper. The proof is given in Section 5. 

\begin{theorem}\label{thm:existence-with-bdd}
    Under Assumption~\ref{asm:BSDE-standing-asm}, if the model parameters are bounded in the mean-field term in the sense that
    \begin{equation}\label{bdd-in-mf}
    \abs{G(x,q^x)} + \abs{B_t(x,0,q^x)} + \abs{F_t(x,0,q)} \leq L \\
    \end{equation}
    for some $L>0$, then the generalized MV-BSDE \eqref{general-MVBSDE} admits a solution.
\end{theorem}

If the model parameters are unbounded in the mean-field terms, more refined conditions are required. In particular, additional growth and boundedness conditions on the drift function are required if the running costs are unbounded in the law of the state-control process. 

\begin{assumption}\label{asm:strong-quad}
The functions $B$ and $F$ satisfy the following conditions for any $(t, x, z, q) \in [0,T] \times \calC_d \times \RR^d \times \calP_1(\calC_d \times \RR^d)$. 
\begin{enumerate}
    \item \label{cond:relax-truncation}
    Either of the following conditions holds for some $\gamma > 0$:
    \begin{enumerate}
        \item The function $F$ is bounded in the law of state;
        \[
        \abs{F_t(x,0,q)} \leq \gamma (1 + M_1 (q^z)).
        \]
        \item \label{cond:unbdd-in-mx}
        The functions $B$ and $F$ satisfy the linear growth conditions;
        \begin{align*}
            \abs{(\sigma B)_t (x,z,q^x)}  \leq \gamma (1 + \abs{z} + M_1(q^x)), \quad
            \abs{F_t(x,0,q)}  \leq \gamma (1 + M_1(q)).
        \end{align*}    
    \end{enumerate}
    \item \label{cond:strong-quad}
    The function $F$ satisfies the strictly quadratic growth condition; for some $\tilde \gamma>0$ either 
    \begin{gather*}
        F_t(x,z,q) \leq -\frac{\tilde{\gamma}}{2} \abs{z}^2 + F_t(x,0,q)
    \quad \mbox{or} \quad 
        F_t(x,z,q) \geq \frac{\tilde{\gamma}}2 \abs{z}^2 - F_t(x,0,q).
    \end{gather*}
\end{enumerate}
\end{assumption}

The following is the main result of this paper. The proof is again given in Section 5.

\begin{theorem}\label{thm:relax-bdd}
Under Assumptions~\ref{asm:BSDE-standing-asm} and \ref{asm:strong-quad}, the generalized MV-BSDE \eqref{general-MVBSDE} admits a solution
 \[
     (X,Y,Z,\barPP) \in \SS^2 (\PP) \times \SS^\infty (\PP) \times \HH^2_{{\operatorname{BMO}}} (\PP) \times \calP(\Omega).
 \]
\end{theorem}


\subsubsection{Generalized MV-BSDEs and MV-FBSDEs}
We proceed with an equivalence result between generalized MV-BSDEs and MV-FBSDEs that will be important for our subsequent analysis. The equivalence is intuitive and is often implicitly assumed in the literature. However, there is an important subtlety on which we now elaborate. The MV-BSDE \eqref{general-MVBSDE} under $\PP$ can be rewritten into the MV-FBSDE under $\barPP$ (see also Remark~\ref{rmk:Girsanov}): 
\begin{equation}\label{general-MVFBSDE}
\begin{split}
    \dif X_t &= (\sigma B)_t (X, Z_t, \barcalL(X)) \dif t + \sigma_t(X) \dif \bar{W}_t, \quad X_0 = x_0, \\
    \dif Y_t &= -F_t(X, Y_t, Z_t, \barcalL(X, Z_t)) \dif t + Z_t \dif \bar{W}_t, \quad Y_T = G(X,\barcalL(X)).
\end{split}
\end{equation}
where 
\[
    \bar{W}_t := W_t - \int_0^t B_s (X, Z_s, \barcalL(X)) \dif s.
\]
This suggests that solving the generalized MV-BSDE \eqref{general-MVBSDE} is equivalent to solving the MV-FBSDE \eqref{general-MVFBSDE}. The subtle issue is the different solution spaces; solving the generalized MV-BSDE under $\PP$ requires integrability conditions under $\PP$ while solving the associated MV-FBSDE requires integrability under $\barPP$. The next proposition shows that the two equations are indeed equivalent under the assumptions made in this section. 

\begin{proposition}\label{prop:equiv-sol-P-barP}
For any measure $\barPP \in \calP(\Omega)$ and any triple $(X,Y,Z)$ of $\FF$-adapted processes, the following conditions are equivalent under Assumption~\ref{asm:BSDE-standing-asm}:
\begin{enumerate}
    \item the quadruple $(X,Y,Z,\barPP)$ solves the McKean-Vlasov BSDE \eqref{general-MVBSDE} and satisfies
    \[
    (X,Y,Z) \in \SS^2(\PP) \times \SS^\infty(\PP) \times \HH^2_{\operatorname{BMO}}(\PP),
    \]
    \item the quadruple $(X,Y,Z,\barPP)$ solves the McKean-Vlasov FBSDE \eqref{general-MVFBSDE} and satisfies
    \[
    (X,Y,Z) \in \SS^2(\barPP) \times \SS^\infty(\barPP) \times \HH^2_{\operatorname{BMO}}(\barPP),
    \]
\end{enumerate}
If a quadruple $(X,Y,Z,\barPP)$ satisfies condition (1) or (2), then for any constants $L, \bar{L} > 0$ there exist positive constants $\gamma_\PP, \gamma_{\barPP} > 0$ such that
\begin{equation}\label{equiv-norm-Z}
\begin{aligned}
    \norm{Z}_{\HH^2_{\operatorname{BMO}}(\PP)} \leq L
    &\quad \Rightarrow \quad
    \norm{Z}_{\HH^2_{\operatorname{BMO}}(\barPP)} \leq \gamma_\PP \norm{Z}_{\HH^2_{\operatorname{BMO}}(\PP)}, \\
    \norm{Z}_{\HH^2_{\operatorname{BMO}}(\barPP)} \leq \bar{L}
    &\quad \Rightarrow \quad
    \norm{Z}_{\HH^2_{\operatorname{BMO}}(\PP)} \leq \gamma_{\barPP}\norm{Z}_{\HH^2_{\operatorname{BMO}}(\barPP)}.
\end{aligned}
\end{equation}
\end{proposition}
\begin{proof}
Since $\PP \sim \barPP$ we only need to verify that the pair $(X,Z)$ belongs to the right spaces.

{\textsl{(1) $\Rightarrow$ (2)}} 
Since $X \in \SS^p(\PP)$ for any $p \in [2,\infty)$, the reverse H\"older inequality yields
\[
\bar{\EE} \left[\sup_{t \in [0,T]} \abs{X_t}^2 \right]
\leq \EE \left[ \abs{\dif \barPP \over \dif \PP}^q \right] \EE \left[ \sup_{t \in [0,T]} \abs{X_t}^{2p} \right]
< \infty
\]
for some H\"older conjugates $p, q > 1$. By Assumption~\ref{asm:BSDE-standing-asm}, 
\[
    \abs{B_t (X, Z_t, \barcalL(X))}
    \leq \abs{B_t(X, 0, \delta_0)} + K (\abs{Z_t} + \bar{\EE}\left[\abs{X_t^\ast} \right] )
    \leq C(1 + \abs{Z_t})
\]
for some constant $C > 0$ and so the process $B_\cdot (X, Z_\cdot, \barcalL(X))$ belongs to $\HH^2_{\operatorname{BMO}}(\PP)$.
Therefore, by Proposition~\ref{prop:BMO-norm-equivalence}, there exists a constant $\gamma_\PP$ that depends only on $\norm{Z}_{\HH^2_{\operatorname{BMO}}(\PP)}$ such that
\[
\norm{Z}_{\HH^2_{\operatorname{BMO}}(\barPP)} \leq \gamma_\PP \norm{Z}_{\HH^2_{\operatorname{BMO}}(\PP)}.
\]

\textsl{(2) $\Rightarrow$ (1)}
To prove the other direction, we use that 
\begin{align*}
\frac{\dif \PP}{\dif \barPP}
&= \calE\left( - B_\cdot (X, Z_\cdot, \barcalL(X)) \cdot \bar{W} \right)_T.
\end{align*}
From H\"older's inequality,  for some H\"older conjugates $p,q > 1$, 
\[
    \bar{\EE}\left[\sup_{t \in [0,T]} \abs{X_t}^r \right] 
    \leq \EE\left[ \abs{\dif \barPP \over \dif \PP}^q \right] \left[\sup_{t \in [0,T]} \abs{X_t}^{rp} \right]
    < \infty
    \quad \text{for all}~r \in [2, \infty).
\]
We can now use the same arguments as above to show that $X \in \SS^2(\PP)$ and $Z \in \HH^2_{\operatorname{BMO}}(\PP)$. 
\end{proof}


\subsection{Our approach}
\label{subsec:our-approach}

To prove the existence of mean-field equilibria it would be natural to consider, for any measure $q \in \calM_1$, the BSDEs \begin{equation} \label{BSDE-mu-q}
\dif Y_t^q = -H_t(X, Z_t^q, q_t) \dif t + Z_t^q \dif W_t, \quad Y_T^q = G(X,q^x_T),
\end{equation}
establish the existence and uniqueness of solutions $(Y^{q}, Z^{q}) \in \SS^2(\PP) \times \HH^2_{\operatorname{BMO}}(\PP)$, consider the probability measures $\PP^{q}$ defined in terms of the processes $Z^{q}$ by
\[
	\frac{\dif \PP^{q}}{\dif \PP} = \calE\left( B_\cdot (X, Z^{q}_\cdot, q^x_\cdot) \cdot  W \right)_T, 
\]
and to prove that the following solution mapping has a fixed-point:
\[
	q \mapsto {\phi}(q) := \calL^{q}(X, Z_\cdot^{q}),
\quad \text{where} \quad \calL^{q}(\cdot) := \PP^{q} \circ (\cdot)^{-1}.
\]

This would require $\phi$ to map a suitable compact, convex set into itself. Since we cannot expect the processes $Z^{q}$ to be continuous (c.f.\ \cite[Section 5]{imkeller2010path}), verifying compactness in $\calM_1$ would be difficult. To bypass this problem, we lift the solution map to the space 
\[
    \calP_1 (\calC_d)  \times \calY_1\big( [0,T] \times \calP_1(\calC_d \times \RR^d) \big) =: \calP_1 \times \calY_1
\]
where $\calY_1([0,T] \times \calP_1(\calC_d \times \RR^d))$ denotes the space of integrable Young measures on $[0,T] \times \calP_1 (\calC_d \times \RR^d)$, equipped with the stable topology.\footnote{Basic properties of (integrable) Young measures are reviewed in Appendix~\ref{appx:Young-measure}.} 

A subtle issue is that in our setting the mapping $z \cdot B_t(x,z,q^x)$ is not necessarily of linear growth in $q^x$ due to a cross term of $z$ and $q^x$, while $F_t(x,z,q)$ is always of linear growth in $q$. This difference turns out to be important when lifting the solution map. To be able to treat the laws in the drift and cost term differently, we introduce the modified Hamiltonian 
\[
    \bar{H}_t(x,z,\mu,q) := F_t(x,z,q) + z \cdot B_t(x,z,\mu)
\]
where $\mu \in \calP_1(\calC_d)$ and $q \in \calP_1(\calC_d \times \RR^d)$. 
Under Assumption~\ref{asm:BSDE-standing-asm}, 
we can then consider, for any law $\mu \in \calP_1$ and Young measure $\boldsym{\nu}(\dif t, \dif q) = \nu_t (\dif q) \dif t \in \calY_1$, 
the quadratic BSDE 
\begin{align} \label{BSDE-mu-nu}
    Y_t^{\mu, \boldsym{\nu}} &= G(X,\mu) + \int_t^T  \int_{\calP_1(\calC_d \times \RR^d)} \bar{H}_s(X, Z_s^{\mu,\boldnu}, \mu, q) \nu_s (\dif q)   \dif s - \int_t^T Z_s^{\mu, \boldsym{\nu}} \dif W_s.
\end{align}
The lifted BSDE is well defined because
\[
\int_0^T \int_{\calP_1(\calC_d \times \RR^d)} \abs{\bar{H}_s(X,0,\delta_0,q)} \nu_s(\dif q) \dif s
\le C \left(1 + \int_0^T \int_{\calP_1(\calC_d \times \RR^d)} M_1(q) \boldnu(\dif q, \dif t) \right) < \infty
\]
for some constant $C > 0$. If the density 
\begin{equation}\label{prob-mu-nu}
    \frac{\dif \PP^{\mu, \boldsym{\nu}}}{\dif \PP} = \calE(B_\cdot(X, Z^{\mu,\boldsym{\nu}}_\cdot, \mu) \cdot W)_T
\end{equation}
is well defined, then we can define the solution map $\Phi$ on $\calP_1 \times \calY_1$ by
\begin{equation}\label{def-sol-map}
\Phi(\mu, \boldsym{\nu}) :=  \big( \calL^{\mu, \boldsym{\nu}}(X), \delta_{\calL^{\mu,\boldsym{\nu}} (X, Z_t^{\mu, \boldsym{\nu}})}(\dif q)\dif t \big), \quad
\text{where} \quad \calL^{\mu,\boldsym{\nu}}(\cdot) := \PP^{\mu,\boldsym{\nu}} \circ (\cdot)^{-1}.
\end{equation}

If that mapping admits a fixed-point $(\mu^*,\boldsym{\nu}^*) \in \calP_1 \times \calY_1$ then our MV-BSDE admits a solution and hence our MFG admits an equilibrium in weak formulation. 

\begin{remark}
We emphasize that the use of {\sl integrable} Young measures is important to account for the integrability condition $\calL^{\mu^*,\boldsym{\nu}^*}(X, \hat \alpha) \in \calM_1$ of the equilibrium control $\hat \alpha$.  
\end{remark}

The following proposition shows that the mapping $\Phi$ is well defined in our setting; the proof follows essentially from \cite[Proposition 3]{briand2008quadratic}. It also provides a uniform BMO-bound if the coefficients are bounded in the mean-field terms, which will be key to our subsequent analysis. 

\begin{proposition}\label{prop:solution-map}
Under Assumption~\ref{asm:BSDE-standing-asm}, the BSDEs \eqref{BSDE-mu-nu} admit a unique solution 
\[
    (Y^{\mu,\boldsym{\nu}}, Z^{\mu,\boldsym{\nu}}) \in \SS^\infty(\PP) \times \HH_{\operatorname{BMO}}^2(\PP)
\]
for each $(\mu,\boldsym{\nu}) \in \calP_1 \times \calY_1$. The solutions satisfy
\[
\norm{Z^{\mu,\boldnu}}_{\HH^2_{\operatorname{BMO}}(\PP)} \le \phi\left( M_1(\mu), \int_0^T \int_{\calP_1(\calC_d \times \RR^d)} M_1(q) \boldnu(\dif q, \dif t) \right)
\]
for some positive function $\phi$.
If the boundedness condition \eqref{bdd-in-mf} holds, then the processes $Z^{\mu,\boldsym{\nu}}$ are uniformly bounded in BMO-norm: 
\begin{equation}\label{Z-unibdd-nu}
    C_{\operatorname{BMO}}:= \sup_{\mu \in \calP_1, \boldsym{\nu} \in \calY_1} \norm{Z^{\mu,\boldsym{\nu}}}_{\HH_{\operatorname{BMO}}^2(\PP)} < \infty.
\end{equation}
\end{proposition}

If our model parameters are bounded in mean-field terms, the uniform BMO bound on the processes $Z^{\mu,\boldsym{\nu}}$ allows us to identify a convex set that the solution map maps onto itself. In fact, for any measure $q=(q^x,q^z)$ in the range of $\Phi$, 
\[
	\int_0^T \int_{\RR^d} |w|^2 \dif q_t^z (w) \dif t \leq C_{\operatorname{BMO}}.
\]
Furthermore, by construction $q^x \ll  \mu_X := \PP \circ X^{-1} $, and by the reverse H\"older inequality there exists a constant $\delta_{\operatorname{RH}} > 0$ such that 
\[
	C_{\operatorname{RH}} := \sup_{\mu \in \calP(\calC_d), \boldsym{\nu} \in \calY} \EE \left[ \calE \left( B_\cdot (X, Z_\cdot^{\mu, \boldsym{\nu}}, \mu) \cdot W \right)_T^{1 + \delta_{\operatorname{RH}}}  \right] < \infty.
\]
In terms of these bounds, we now introduce the sets 
    \begin{equation}\label{domain}
    \begin{aligned}
    \calQ^* &:= \left\{\mu \in \calP_1 (\calC_d) ~\middle|~ \mu \ll \mu_X,~ \int_{\calC_d} \abs{\frac{\dif \mu}{\dif \mu_X}}^{1 + \delta_{\operatorname{RH}}} \dif \mu_X \leq C_{\operatorname{RH}} \right\},\\
    \calK^o
    &:= \left\{
      q \in \calM_1 ~\middle|~ 
      q_t^x \in \calQ^\ast ~ \mbox{for all } t \in [0,T], ~
      \int_0^T \int_{\RR^d} \abs{w}^2 \dif q_t^z (w) \dif t \leq C_{\operatorname{BMO}}
    \right\}.
    \end{aligned}
    \end{equation}

The set $\calQ^*$ is a convex subset of $\calP_1 (\calC_d) $ and $\calK^o$ is a convex subset of $\calM_1$. When embedded into the set $\calY_1$, the set $\calK^o$ may not be convex, but the following set is: 
\[
	\calK^* := \overline{\operatorname{conv}\{\delta_{q_t}(\dif q) \dif t \in \calY_1 \mid q \in \calK^o\}}^{\mathcal{S}_1(\calP_1(\calC_d \times \RR^d))}.
\]

\begin{corollary}\label{cor:maps-onto-itself}
	Under Assumption~\ref{asm:BSDE-standing-asm} and the boundedness condition \eqref{bdd-in-mf}, the solution mapping $\Phi$ can be viewed as a mapping
	\[
		\Phi : \calP_1 \times \calY_1 \to \calQ^\ast \times \calK^*. 
	\]
	In particular, it maps the closed convex set $\calQ^\ast \times \calK^*$ to itself. 
\end{corollary}

In Section 3 we prove that the solution mapping is continuous and in Section 4 that $\calQ^* \times \calK^*$ is compact. As a result, the solution mapping has a fixed point $(\mu^*,\boldsym{\nu}^*) \in \calP_1 \times \calY_1$ and hence a MFG equilibrium in weak formulation exists when the model parameters are bounded in mean-field terms. The case of unbounded coefficients will be solved by approximation.


\subsection{Examples}
\label{subsec:examples}

We close this section with a discussion of two toy models that are covered by our framework but not by existing ones. For simplicity, we consider only the case of one-dimensional state dynamics and control processes. The below examples easily extend to multi-dimensional settings and  time-delayed models where $x_t$ is replaced by $x_{t-\delta}$ for some $\delta > 0$.

\begin{example}[State process governed by Geometric Brownian motion]\label{eg:bounded-in-qx}
Let
\begin{gather*}
x_0 > 0, \quad
\sigma_t (x) := x_t, \quad
b_t (x,m^x,a) := x_t \left\{ a + \int_{\calC_1} \bar{x}_t m^x (\dif \bar{x}) \right\}, \quad g(x,m^x) := \int_{\calC_1} \bar{x}_T m^x (\dif \bar{x})
\end{gather*}
and
\[
f_t (x,m,a) := \left\{ -\frac12 a + \phi \left( \int_{\calC_1} \bar{x}_t  m^x (\dif \bar{x}) \right) \right\} \cdot a + \bar{f} (x_t) \int_\RR \bar{a} m^a (\dif \bar{a}),
\]
where $\phi:\RR \to \RR$ is bounded and continuous, and $\bar{f}:\RR \to \RR$ is measurable and bounded. In this case $X$ is geometric Brownian motion, and hence $\sigma_t (X)$ is a.s.~invertible.
Moreover, the maximizer $\Lambda$ of the reduced Hamiltonian \eqref{Hamiltonian} is given by
\begin{equation} \label{Lambda-eg}
\Lambda_t (X, m, z)
=  z + \phi \left(\int_{\calC_1} \bar{x}_t m^x (\dif \bar{x}) \right).
\end{equation}
Since the growth condition \eqref{item:BMO-sigma-b-Lambda} of Assumption~\ref{asm:MFG-growth-conditions} is assumed on $\sigma^{-1}b$, not separately on $\sigma^{-1}$ and $b$, our framework  accommodates geometric Brownian motion as state process. This point has already been emphasized in  \cite[Example 5.7 ]{Carmona2015}. However, their result does not cover this case as they require  $\sigma^{-1}b$ to be bounded, in which case their maximizer $\Lambda_t$ is also bounded.
If $\Lambda$ is bounded, then the Hamiltonian is Lipschitz continuous in $z$, a case that has already been addressed in, e.g.~ \cite{possamai2025non}.
Furthermore, our drift function $b$ depends on the product of the state and control variables. Control problems involving such product term are often challenging. To the best of our knowledge, there are very few universal results for these problems.
\end{example}

\begin{example}[Cost function unbounded in the law of state and control]\label{eg:unbounded-in-qx}
Let
\begin{gather*}
x_0 > 0, \quad
\sigma_t (x) := 1, \quad
b_t (x,m^x,a) := a + \int_{\calC_1} \bar{x}_t  m^x (\dif \bar{x}), \quad
g(x,m^x) := \int_{\calC_1} \bar{x}_T m^x (\dif \bar{x}),
\end{gather*}
and
\[
f_t (x,m,a) := \left\{ - \frac12 a + \phi \left( \int_{\calC_1} \bar{x}_t  m^x (\dif \bar{x}) \right) \right\} \cdot a + \bar{f} (x_t) \int_{\calC_1 \times \RR} (\bar{x}_t + \bar{a}) m (\dif \bar{x}, \dif \bar{a}),
\]
where 
$\phi$ and $\bar{f}$ are as in Example~\ref{eg:bounded-in-qx}.
We see that the maximizer $\Lambda$ of the reduced Hamiltonian is given by \eqref{Lambda-eg}. Unlike in Example~\ref{eg:bounded-in-qx}  the running cost function is unbounded in the law of the state. This comes at the expense of dropping the multiplicative term $x \cdot a$ from the drift function. The drift term in Example~\ref{eg:bounded-in-qx} does not meet Condition \eqref{cond:unbdd-in-mx} in Assumption~\ref{asm:strong-quad}. 
\end{example}


\section{The solution mapping}
\label{sec:stability}

In this section, we prove the continuity of the solution mapping $\Phi: \calP_1 \times \calY_1 \to \calP_1 \times \calY_1$. Our continuity result is based on a novel stability result for quadratic MV-BSDEs. 

\subsection{Stability of quadratic MV-BSDEs}

The proof of the stability result requires a series of auxiliary results that we now present. In what follows $(E, \norm{\cdot}_E)$ is a Banach space, and for an indexed measure $\PP^i \in \calP(\Omega)$ we set $\EE^i [\cdot] := \EE^{\PP^i}[\cdot]$ and $\calL^i (\cdot) := \PP^i \circ (\cdot)^{-1}$.

\begin{proposition}\label{prop:stability-Young-meas}
    Let $\PP^n \in \calP(\Omega)$ and $X^n \in \HH^2(E; \PP^n) \cap \HH^2(E;\PP^\infty)$ for any $n \in \NN^\ast$, satisfy
    \begin{equation}\label{uniform-H2-bound}
    \textstyle
        \sup_{n \in \NN} \big\{\norm{X^n}_{\HH^2(E;\PP^n)} + \norm{X^n}_{\HH^2(E;\PP^\infty)} \big\} < \infty.
    \end{equation}
    If
    \begin{equation}\label{TV-convergence-PPn}
    \norm{\PP^n - \PP^\infty}_{\operatorname{TV}} \to 0
    \end{equation}
    and
    \begin{equation}\label{L2-convergence-PPinfty}
        \EE^\infty \left[ \int_0^T \norm{X_s^n - X_s^\infty}_E^2 \dif s \right] \to 0,
    \end{equation}
    then 
    \begin{equation}\label{general-stable-convergence}
        \delta_{\calL^n (X_t^n)}(\dif q)\dif t \to \delta_{\calL^\infty (X_t^\infty)}(\dif q) \dif t \quad \text{in} \quad \calS_1(\calP_1(E)).
    \end{equation}
\end{proposition}
\begin{proof}
In view of Proposition~\ref{prop:Portmanteau-thm-intYoung}, it suffices to prove that
\begin{equation*}\label{convergence-integrated-W1}
    \int_0^T \calW_1(\calL^n(X_t^n), \calL^\infty(X_t^n)) \dif t \to 0, \quad \mbox{and} \quad
    \int_0^T \calW_1(\calL^\infty(X_t^n), \calL^\infty(X_t^\infty)) \dif t \to 0.
\end{equation*}
The second convergence follows directly from \eqref{L2-convergence-PPinfty}. To establish the first, we set 
\[
    Y_t^{n,R} := X_t^n \boldsym{1}_{\{\norm{X_t^n}_E \le R\}} \quad \mbox{for} \quad R>0, ~ n \in \NN^\ast.
\]
Then, 
\begin{align*}
    \calW_1&(\calL^n(X_t^n), \calL^\infty(X_t^n)) \\
    &\le \underbracket{\calW_1 \big( \calL^n(Y_t^{n,R}), \calL^\infty(Y_t^{n, R}) \big)}_{=: D_t^1(n,R)}
    + \underbracket{\calW_1 \big( \calL^n(X_t^n), \calL^n(Y_t^{n,R}) \big) + \calW_1 \big( \calL^\infty(X_t^n), \calL^\infty(Y_t^{n,R}) \big)}_{=:D_t^2(n,R)}.
\end{align*}
The term $D_t^1(n,R)$ satisfies $D_t^1(n,R) \le 2R \norm{\PP^n - \PP^\infty}_{\operatorname{TV}}$, 
and the term $D_t^2(n,R)$ satisfies 
\begin{equation}\label{controlled-by-second-moment}
\begin{aligned}
    D_t^2(n,R)
    &\le \EE^n\left[\norm{X_t^n}_E \boldsym{1}_{\{\norm{X_t^n}_E > R\}}\right] + \EE^\infty \left[\norm{X_t^n}_E \boldsym{1}_{\{\norm{X_t^n}_E > R\}} \right] \\
    &\le \frac{1}{R} \Big( \EE^n \left[\norm{X_t^n}_E^2] + \EE^\infty[\norm{X_t^n}_E^2 \right] \Big).
\end{aligned}
\end{equation}
Combining the two estimates, we obtain the following:
\[
\int_0^T \calW_1 (\calL^n(X_t^n), \calL^\infty(X_t^n)) \dif t
\le 2RT \norm{\PP^n - \PP^\infty}_{\operatorname{TV}} + \frac{C}R,
\]
for any $R > 0$ and $n \in \NN$, where $C := \sup_{n \in \NN} \big( \norm{X^n}_{\HH^2(E;\PP^n)}^2 + \norm{X^n}_{\HH^2(E;\PP^\infty)}^2 \big) < \infty$ by assumption \eqref{uniform-H2-bound}. Taking $R = \norm{\PP^n - \PP^\infty}_{\operatorname{TV}}^{-1/2}$ yields \[
\int_0^T \calW_1 (\calL^n(X_t^n), \calL^\infty(X_t^n)) \dif t
\le (2T + C) \sqrt{\norm{\PP^n - \PP^\infty}_{\operatorname{TV}}}.
\]
\end{proof}

When we later apply Proposition~\ref{prop:stability-Young-meas} to our stability result (Theorem~\ref{thm:general-stability}), the probability measures $\PP^n$ in Proposition~\ref{prop:stability-Young-meas} will have densities $\frac{\dif \PP^n}{\dif \PP}$ w.r.t.~ $\PP$ that are determined by stochastic exponentials.
The next proposition shows that, under this condition, the $L^2$-convergence \eqref{L2-convergence-PPinfty} implies the convergence \eqref{TV-convergence-PPn} in total variation norm. A related result has previously been obtained by 
Carmona and Lacker \cite{carmona_lacker_2015} using the Pinsker inequality. We prepare the proposition with the following basic lemma that will be used repeatedly throughout the rest of the paper.

\begin{lemma}\label{lem:uni-bdd-norm}
    Let $X:\Omega \to E$ be a random variable and $I$ be an index set. For any $i \in I$ let $Z^i \in \HH_{\operatorname{BMO}}^2(\RR^d;\PP)$, let $\theta^i : [0,T] \times \Omega \times \RR^d \to \RR^d$ be a measurable function that satisfies
    \begin{equation}\label{Lip-conti-theta}
        \abs{\theta_t^i(z) - \theta_t^i(\bar{z})} \leq K_{\theta} |z-\bar{z}|, \quad K_\theta < \infty,
    \end{equation}
    and let $\PP^i \in \calP(\Omega)$ be a probability measure with density $\frac{\dif \PP^i}{\dif \PP} = \calE\left(\theta^i(Z^i) \cdot  W \right)_T.$
    If 
    \begin{equation}\label{uni-bdd-Z-theta}
        X \in \bigcap_{p < \infty} \LL^p(E;\PP) \quad \mbox{and} \quad
         \sup_{i \in I} \norm{Z^i}_{\HH_{\operatorname{BMO}}^2(\PP)} + \sup_{i \in I} \norm{\theta^i(0)}_{\HH_{\operatorname{BMO}}^2(\PP)} < \infty,
    \end{equation}
    then 
    \begin{equation}\label{uni-bdd-density}
        \sup_{i \in I} \int_E \abs{\frac{\dif (\PP^i \circ X^{-1})}{\dif (\PP \circ X^{-1})}}^{1+\delta} \dif (\PP \circ X^{-1}) < \infty
    \end{equation}
    for some $\delta > 0$.
    Furthermore, 
    \begin{equation}\label{uni-bdd-density2}
        \sup_{i \in I} \norm{X}_{\LL^2(E; \PP^i)} < \infty
    \end{equation}
    and for any $i_0 \in I$,
    \begin{equation}\label{uni-bdd-norm}
        \sup_{i \in I} \Big\{ \norm{Z^i}_{\HH^2_{\operatorname{BMO}}(\PP^i)} + \norm{Z^i}_{\HH^2_{\operatorname{BMO}}(\PP^{i_0})} \Big\} < \infty.
    \end{equation}
\end{lemma}
\begin{proof}
By \cite[Lemma 7.6]{Carmona2015}\label{lem:density-equal-condEXP} and the Jensen inequality, it holds for any $\delta > 0$ that
\begin{align*}
    \int_{E} \abs{\frac{\dif (\PP^i \circ X^{-1})}{\dif (\PP \circ X^{-1})}}^{1+\delta} \dif (\PP \circ X^{-1})
    &= \EE \left[ \abs{\frac{\dif (\PP^i \circ X^{-1})}{\dif (\PP \circ X^{-1})} \circ X }^{1+\delta} \right] 
     = \EE\left[ \EE \left[\frac{\dif \PP^i}{\dif \PP} \,\middle|\, X\right]^{1+\delta} \right] \\
    &\leq \EE \left[ \calE \big( \theta^i(Z^i) \cdot W \big)_T^{1+\delta} \right].
\end{align*}
Assumptions \eqref{Lip-conti-theta} and \eqref{uni-bdd-Z-theta} imply $\sup_{i \in I} \norm{\theta^i(Z^i)}_{\HH_{\operatorname{BMO}}^2(\PP)} < \infty$ and so \eqref{uni-bdd-density} follows from the reverse H\"older inequality. The inequality \eqref{uni-bdd-norm} follows from the uniform BMO-bound on $\theta^i(Z^i)$ along with Proposition~\ref{prop:BMO-norm-equivalence} which yields $C>0$ that does not depend on $i \in I$ such that 
\begin{align*}
     \norm{Z^i}_{\HH_{\operatorname{BMO}}^2(\PP^i)} + \norm{Z^i}_{\HH_{\operatorname{BMO}}^2(\PP^{i_0})}
     \leq C \norm{Z^i}_{\HH_{\operatorname{BMO}}^2(\PP)}.
\end{align*}
Finally, \eqref{uni-bdd-density2} follows from the fact that for some $q > 1$,
\begin{align*}
    \EE^i [ \norm{X}_E^2]
    &\le \int_E \norm{x}_E \abs{\frac{\dif (\PP^i \circ X^{-1})}{\dif (\PP \circ X^{-1})}} (\PP \circ X^{-1})(\dif x) \\
    & \le \EE[{\norm{X}_E^q}]^{\frac1q} \int_{E} \abs{\frac{\dif (\PP^i \circ X^{-1})}{\dif (\PP \circ X^{-1})}}^{1+\delta} \dif (\PP \circ X^{-1}).
\end{align*}
\end{proof}

\begin{proposition}\label{prop:TV-convergence}
Under the assumptions of Lemma~\ref{lem:uni-bdd-norm} with $I = \NN^\ast$ if 
\begin{equation}\label{convergence-asm-theta}
\EE \left[ \left( \int_0^T \abs{(\theta^n_s - \theta^\infty_s)(Z_s^\infty)}^2 \dif s \right)^p \right] \to 0 \quad \text{for any}~p \in [1,\infty),
\end{equation}
then \eqref{L2-convergence-PPinfty} implies \eqref{TV-convergence-PPn} with $X^n = (X, Z^n)$. In particular, the convergence \eqref{L2-convergence-PPinfty} implies \eqref{general-stable-convergence}.
\end{proposition}
\begin{proof}
Condition \eqref{uniform-H2-bound} is satisfied from Lemma~\ref{lem:uni-bdd-norm}.
In view of Pinsker's inequality, the convergence \eqref{TV-convergence-PPn} follows if $H(\PP^\infty \mid \PP^n) \to 0$, where $H$ is the Kullback information. Then,
\begin{align*}
    H(\PP^\infty \mid \PP^n)
    &= -\EE^\infty \left[\log \frac{\dif \PP^n}{\dif \PP^\infty} \right] \\
    &\leq K_\theta^2 \EE^\infty \left[ \int_0^T \abs{Z_s^n - Z_s^\infty}^2 \dif s \right] + \EE^\infty \left[ \int_0^T \abs{(\theta_s^n - \theta_s^\infty)(Z_s^\infty)}^2 \dif s \right].
\end{align*}
By the reverse H\"older inequality there exist $p \in (1,\infty)$ and $C_p > 0$ such that
\[
    \EE^\infty \left[ \int_0^T \abs{(\theta_s^n - \theta_s^\infty)(Z_s^\infty)}^2 \dif s \right]
    \leq C_p \EE^\infty \left[ \left( \int_0^T \abs{(\theta_s^n - \theta_s^\infty)(Z_s^\infty)}^2 \dif s \right)^p \right].
\]
Thus, the convergence \eqref{L2-convergence-PPinfty} implies \eqref{TV-convergence-PPn}, due to the assumption \eqref{convergence-asm-theta}.
\end{proof}

In view of the preceding proposition, our stability result for quadratic BSDEs reduces to verifying the condition \eqref{L2-convergence-PPinfty}. For this, the following lemma will be useful. 

\begin{lemma}\label{lem:convergence-in-prob}
    A sequence $\{X^n\}_{n \in \NN}$ of $E$-valued random variables on $(\Omega, \calF, \PP)$ converges to a random variable $X^\infty$ in probability if there exists a BMO-bounded sequence of stochastic processes $\{\Psi^n\}_{n \in \NN} \subset \HH_{\operatorname{BMO}}^2(\PP)$ such that 
    \begin{equation}\label{conditions-Psyn}
        \EE^{\Psi^n} \big[ \norm{X^n - X^\infty}_E \big] \to 0 \quad \mbox{where} \quad 
        \dif \PP^{\Psi^n} := \calE(\Psi^n \cdot W)_T \dif \PP.
    \end{equation}
\end{lemma}
\begin{proof}
Using the H\"older inequality, we have for any $p \in (1,\infty)$ that
\begin{align*}
\EE \left[ \norm{X^n -X^\infty}_E^{\frac1p} \right]
&= \EE \left[ \calE(\Psi^n \cdot W)_T^{\frac1p} \norm{X^n -X^\infty}_E^{\frac1p} \calE(\Psi^n \cdot W)_T^{-\frac1p} \right] \\
&\leq  \EE^{\Psi^n} \left[ \norm{X^n -X^\infty}_E\right]^{\frac1p} \EE \left[ \left( \frac1{\calE(\Psi^n \cdot W)_T} \right)^{\frac1{q-1}} \right]^{\frac1q},
\end{align*}
where $q$ is the H\"older conjugate of $p$.
Since the assumption \eqref{conditions-Psyn} implies the uniform boundedness of $\Psi^n$ in $\HH^2_{\operatorname{BMO}}(\PP)$, the reverse H\"older inequality yields $q_2 \in (1,\infty)$ such that
\[
\sup_{n \in \NN} \EE \left[ \left( \frac1{\calE(\Psi^n \cdot W)_T} \right)^{\frac1{q_2-1}} \right] < \infty.
\]
Therefore, there exists a constant $C > 0$ that does not dependent on $n \in \NN$ such that
\begin{equation*}
\EE \left[\norm{X^n -X^\infty}_E^{\frac1{p_2}} \right]
\leq C \EE^{\Psi^n} \big[ \norm{X^n -X^\infty}_E \big]^{\frac1{p_2}},
\end{equation*}
where $p_2$ is the conjugate of $q_2$.
The desired convergence therefore follows from \eqref{conditions-Psyn}.
\end{proof}

We are now ready to state our stability result for quadratic BSDEs. 

\begin{theorem}\label{thm:general-stability}
For any $n \in \NN^\ast$ let $\xi^n : \Omega \to \RR$ be a random variable and $$H^n:[0,T] \times \Omega \times \RR \times \RR^d \to \RR$$ be a measurable function such that $H^n_\cdot(\cdot,y,z)$ is $\FF$-progressively measurable for any $(y,z)$, and
the (local) Lipschitz continuity condition
\begin{equation}\label{stability-asm-Lip}
    \abs{H_t^n(y,z) - H_t^n(\bar{y}, \bar{z})}
    \leq K_H \{ \abs{y - \bar{y}} + (1 + \abs{z} + \abs{\bar{z}})\abs{z - \bar{z}} \},
\end{equation}
holds for some $K_H > 0$. Assume that the BSDEs 
\begin{equation}\label{general-n-BSDE}
    Y_t^n = \xi^n + \int_t^T H_s^n (Y_s^n, Z_s^n) \dif s - \int_t^T Z_s^n \dif W_s \quad (n \in \NN^\ast)
\end{equation}
admit solutions
\[
    (Y^n, Z^n) \in \bigcap_{p < \infty} \SS^p (\PP) \times \HH_{\operatorname{BMO}}^2(\PP),
\]
and let the family of functions $\theta^n$ and probability measures $\PP^n \in \calP(\Omega)$ be as in Lemma~\ref{lem:uni-bdd-norm} with $I=\NN^\ast$.
If the integrability conditions \eqref{uni-bdd-Z-theta} hold for $(\theta^n, Z^n)$, if 
\begin{equation}\label{stability-asm-integrability}
\begin{split}
    \xi^n \in \bigcap_{p < \infty} L^p(\Omega), \quad
    \sup_{n \in \NN} \EE\left[\abs{\int_0^T \abs{H_s^n(0,0)} \dif s}^p \right] < \infty
    \quad \text{for all}~p \in [0,\infty),
\end{split}
\end{equation}
and if for all $p \in [1,\infty)$ and $t \in [0,T]$, 
\begin{equation}\label{stability-asm-convergence}
\begin{split}
    & \EE \left[ \abs{\xi^n - \xi^\infty}^p + \left( \int_0^T \abs{(\theta^n_s - \theta^\infty_s)(Z_s^\infty)}^2 \dif s \right)^p \right] \to 0, \\
    &\int_t^T (H_s^n - H_s^\infty)(Y_s^\infty, Z_s^\infty) \dif s \to 0 \quad\PP\text{-a.s.}
\end{split}
\end{equation}
as $n \to \infty$, then the convergence \eqref{general-stable-convergence} holds with $X^n = (X,Z^n)$ for any $X \in \bigcap_{p < \infty} \LL^p(E;\PP)$.
\end{theorem}
\begin{proof}
Let $\Delta Z^n := Z^n - Z^\infty$.
Since $\theta^n$ and $Z^n$ satisfy the assumptions in Lemma~\ref{lem:uni-bdd-norm}, the convergence \eqref{general-stable-convergence} holds from Proposition~\ref{prop:TV-convergence} if~\footnote{A stronger convergence result for $(Z^n)$ has been obtained in \cite[Theorem 7.3.4]{zhang2017backward} under a uniform boundedness on the sequence $(Y^n)$ in the $\SS^\infty$-norm. We do not require this condition for \eqref{convergence-DeltaZ-wrt-PPinfty}.}
\begin{equation}\label{convergence-DeltaZ-wrt-PPinfty}
    \EE^\infty \left[\int_0^T \abs{\Delta Z_s^n}^2 \dif s \right] \to 0.
\end{equation}

\hspace{2mm}

\textsl{Step 1.} The convergence \eqref{convergence-DeltaZ-wrt-PPinfty} follows from Vitali's convergence theorem if 
\begin{equation}\label{DeltaZ-uni-ible}
    \left\{ \int_0^T \abs{\Delta Z_s^n}^2 \dif s \right\}_{n \in \NN}~\text{is uniformly integrable w.r.t.}~\PP^\infty,
\end{equation}
and
\begin{equation}\label{DeltaZ-prob-convergence-infty}
    \int_0^T \abs{\Delta Z_s^n}^2 \dif s \to 0
    \quad \text{in probability under}~\PP^\infty.
\end{equation}
By the energy inequality (Proposition~\ref{prop:energy-ineq}) and the equivalency of the BMO norms (Proposition~\ref{prop:BMO-norm-equivalence}), it follows from our assumptions that
\[
\sup_{n \in \NN}\EE^\infty \left[ \left( \int_0^T \abs{\Delta Z_s^n}^2 \dif s \right)^2 \right]
\leq 2 \sup_{n \in \NN} \norm{\Delta Z^n}_{\HH^2_{\operatorname{BMO}}(\PP^\infty)}^4
< \infty.
\]
This implies the uniform integrability \eqref{DeltaZ-uni-ible}.
To prove the convergence \eqref{DeltaZ-prob-convergence-infty}, it suffices to prove the corresponding convergence in probability under $\PP$ since $\PP^\infty \ll \PP$.

\hspace{2mm}

\textsl{Step 2.} By Lemma~\ref{lem:convergence-in-prob}, the convergence \eqref{DeltaZ-prob-convergence-infty} under $\PP$ follows if we find a BMO-bounded sequence $\{\Psi^n\}_{n \in \NN} \subset \HH_{\operatorname{BMO}}^2(\PP)$ that satisfies \eqref{conditions-Psyn} with $X^n := \int_0^T \abs{\Delta Z_s^n}^2 \dif s$ and $X^\infty := 0$.
For this, we use the BSDEs \eqref{general-n-BSDE}. Setting $\Delta \xi^n := \xi^n - \xi^\infty$ and $\Delta Y^n := Y^n - Y^\infty$, we see that
\begin{equation}\label{stability-linearization}
\begin{aligned}
    \Delta Y_t^n 
    &= \Delta \xi^n + \int_t^T \{ \delta_y H_s^n \Delta Y_s^n  + \Delta_n H_s (Y_s^\infty, Z_s^\infty) \} \dif s
    - \int_t^T \Delta Z_s^n \dif \bar{W}_s^n,
\end{aligned}
\end{equation}
where 
\begin{gather*}
    \bar{W}_t^n := W_t - \int_0^t \delta_z H_s^n \dif s, \quad
    \Delta_n H_t (y,z) := H_t^n (y,z) - H_t^\infty (y,z), \\
    \delta_y H_t^n := \frac{H_t^n (Y_t^n, Z_t^n) - H_t^n (Y_t^\infty, Z_t^n)}{\Delta Y_t^n} \mathbf{1}_{\{\Delta Y_t^n \ne 0\}}, \\
    \delta_z H_t^n := \frac{H_t^n (Y_t^\infty, Z_t^n) - H_t^n (Y_t^\infty, Z_t^\infty)}{\abs{\Delta Z_t^n}^2} \cdot \Delta Z_t^n \mathbf{1}_{\{ \abs{\Delta Z_t^n} \ne 0\}}.
\end{gather*}
In particular, $\bar{W}^n$ is a Brownian motion under the probability measure $\barPP^n$ with density $\frac{\dif \barPP^n}{\dif \PP} = \calE(\delta_z H^n \cdot W)_T;$ the density is well defined since $\delta_z H^n \in \HH_{\operatorname{BMO}}^2(\PP)$ by assumption. In the next step, we show that condition \eqref{conditions-Psyn} holds with $\Psi^n := \delta_z H^n$. 

\hspace{2mm}

\textsl{Step 3.}
The uniform BMO bound follows from {the assumption \eqref{uni-bdd-Z-theta}}. It remains to prove 
\begin{equation}\label{DeltaZ-convergence-PPn}
    \bar{\EE}^n\left[ \int_0^T \abs{\Delta Z_s^n}^2 \dif s \right] \to 0
    \quad (n \to \infty).
\end{equation}

The above expected values are well defined since $\Delta Z^n \in \HH_{\operatorname{BMO}}^2(\PP)$ and $\delta_z H^n \in \HH_{\operatorname{BMO}}^2(\PP)$ and so Proposition~\ref{prop:BMO-norm-equivalence} implies $\Delta Z^n \in \HH_{\operatorname{BMO}}^2(\bar{\PP}^n)$. By moving the stochastic integral to the left hand side in the equation \eqref{stability-linearization} and then taking expectations under $\barPP^n$, we see that
\begin{equation}\label{stability-expec-PPn}
\begin{aligned}
\bar{\EE}^n &[\abs{\Delta Y_t^n}^2] + \bar{\EE}^n\left[\int_t^T \abs{\Delta Z_s^n}^2 \dif s \right] \\
&= \bar{\EE}^n \left[ \abs{\Delta \xi^n + \int_t^T \{ \delta_y H_s^n \Delta Y_s^n + \Delta_n H_s (Y_s^\infty, Z_s^\infty) \} \dif s}^2 \right] \\
&\leq 3 \left\{ T\abs{K_H}^2 \int_t^T \bar{\EE}^n[\abs{\Delta Y_s^n}^2] \dif s + \bar{\EE}^n \left[ \abs{\Delta \xi^n}^2 + \abs{\int_t^T \Delta_n H_s(Y_s^\infty, Z_s^\infty)}^2 \right]  \right\}.
\end{aligned}
\end{equation}
From Gronwall's inequality it then follows that
\[
\bar{\EE}^n[\abs{\Delta Y_t^n}^2]
\leq C \left( \bar{\EE}^n[\abs{\Delta \xi^n}^2] + \alpha_t^n + \int_t^T \alpha_s^n \dif s \right)
\]
for some $C>0$, where 
\[
\alpha_t^n := \bar{\EE}^n \left[ \abs{\int_t^T \Delta_n H_s(Y_s^\infty, Z_s^\infty) \dif s}^2 \right].
\]
Plugging this result into the inequality \eqref{stability-expec-PPn} yields for some $\tilde{C} > 0$
\begin{equation}\label{stability-estimate-Z}
\begin{aligned}
\bar{\EE}^n\left[ \int_0^T \abs{\Delta Z_s^n}^2 \dif s \right]
&\leq \tilde{C} \left (\bar{\EE}^n[\abs{\Delta \xi^n}^2] + \alpha_0^n + \int_0^T \alpha_t^n \dif t \right).
\end{aligned}
\end{equation}

Since $\sup_{n \in \NN} \norm{\delta_z H^n}_{\HH^2_{\operatorname{BMO}}(\PP)} < \infty$, the reverse H\"older inequality yields $p_1 \in (1,\infty)$ s.t.
\[
\bar{\EE}^n [\abs{\Delta \xi^n}^2]
\leq \EE[\calE(\delta_z H^n \cdot W)_T^{p_1}]^{\frac1{p_1}} \EE [\abs{\Delta \xi^n}^{2q_1}]^{\frac1{q_1}}
\leq C_{p_1} \EE [\abs{\Delta \xi^n}^{2q_1}]^{\frac1{q_1}},
\]
where $q_1$ is the H\"older conjugate of $p_1$ and $C_{p_1}$ depends only on $p_1$.
Thus, by assumption 
\[
\bar{\EE}^n [\abs{\Delta \xi^n}^2] \to 0 \quad (n \to \infty).
\]
Analogously, we can apply the reverse H\"older inequality to obtain
\[
\alpha_t^n \leq C_{p_1} \EE \left[ \abs{\int_t^T \Delta_n H_s (Y_s^\infty, Z_s^\infty) \dif s}^{2q_1} \right]^{\frac1{q_1}}.
\]
Therefore, $\alpha_t^n \to 0$ for all $t \in [0,T]$. Since $\sup_{n \in \NN} \int_0^T \abs{\alpha_t^n}^{q_1} \dif t < \infty$ the functions are uniformly integrable, hence   
$
\int_0^T \alpha_t^n \dif t \to 0.
$
Thus, the second term also converges to zero.
\end{proof}


\subsection{Continuity of the solution mapping}

Since the space $\calC_d \times \RR^d$ is Banach space, the set $\calP_1(\calC_d \times \RR^d)$ is a Polish space under the $\calW_1$-topology. Thus, by Proposition~\ref{prop:spaceY-metrizable} the stable topology $\calS(\calP_1(\calC_d \times \RR^d))$ on the set of Young measures on $\calP_1(\calC_d \times \RR^d)$ is metrizable. Since this topology is weaker than the $\calS_1(\calP_1(\calC_d \times \RR^d))$-topology when restricted to the set of \textsl{integrable} Young measures, any compact set $\calK \subset \calY_1$ is metrizable. This allows us to prove that the solution mapping $\Phi$ is sequentially continuous and continuous on any compact set.\footnote{We used the following fact. If $f : X \to Y$ is a continuous bijection and if $X$ is compact and $Y$ is Hausdorff, then $f$ is homeomorphism. In particular, any two compact topologies $\tau_0 \subset \tau_1$ on the common set $X$ coincide if $\tau_0$ is Hausdorff; c.f.\ \cite[Theorem 26.6]{munkres2000topology} and \cite[p.62]{rudin1991functional}. 
}

\begin{corollary}\label{cor:FP-continuity}
Under Assumption~\ref{asm:BSDE-standing-asm}, the solution mapping $\Phi: \calP_1 \times \calY_1 \to \calP_1 \times \calY_1$ is sequentially continuous and continuous on any compact subset of $\calP_1 \times \calY_1$.
\end{corollary}
\begin{proof}
In view of the preceding discussion, we only need to establish sequential continuity. Let 
\[
    \{(\mu^n, \boldsym{\nu}^n)\}_{n \in \NN} \subset \calP_1 \times \calY_1
\]
be a sequence  that converges to $(\mu^\infty, \boldsym{\nu}^\infty) \in \calP_1 \times \calY_1$, and let $\nu_\cdot^n$ be the unique disintegration of the Young measure $\boldsym{\nu}^n$. For each $n \in \NN^\ast$, let
\begin{equation}\label{parameters-for-continuity}
\begin{gathered}
\xi^n := G(X,\mu^n), \quad \theta_t^n(z) := B_t(X,z,\mu^n), \quad \PP^n := \PP^{\mu^n,  \boldsym{\nu}^n}, \\
H_t^n(z) := \int_{\calP_1(\calC_d \times \RR^d)} \bar{H}_t(X, z, \mu^n, q) \nu_t(\dif q), \quad
(Y^n, Z^n) := (Y^{\mu^n, \boldsym{\nu}^n}, Z^{\mu^n, \boldsym{\nu}^n}).
\end{gathered}
\end{equation}

We need to verify the assumptions of Theorem~\ref{thm:general-stability}.
The assumptions on $X$ and $\theta^n$ are clear.
Since the sequence $\{(\mu^n, \boldnu^n)\}_{n \in \NN}$ converges in $\calP_1 \times \calY_1$ and hence
\[
\sup_{n \in \NN} M_1(\mu^n) < \infty, \quad
\sup_{n \in \NN} \int_0^T \int_{\calP_1(\calC_d \times \RR^d)} \calW_1(\delta_0, q) \boldnu^n (\dif q, \dif t) < \infty,
\]
the integrability condition \eqref{stability-asm-integrability} follows from Assumption~\ref{asm:BSDE-standing-asm}. The uniform BMO-bound \eqref{uni-bdd-Z-theta} follows from Proposition~\ref{prop:solution-map}.
It remains to verify the convergence \eqref{stability-asm-convergence}. The convergences
\begin{equation}\label{convergence-xi}
\EE[\abs{\xi^n - \xi^\infty}^p] \to 0, \quad
    \EE\left[ \left( \int_0^T \abs{(\theta_s^n - \theta_s^\infty)(Z_s^\infty)}^2 \dif s \right)^p \right]
    \to 0  
    \quad \text{for all}~p \in [1,\infty),
\end{equation}
follow from the dominated convergence theorem.
In addition, for any $t \in [0,T]$, 
\begin{align*}
    \int_t^T (H_s^n - H_s^\infty)(Z_s^\infty) \dif s
    &= \int_t^T \int_{\calP_1(\calC_d \times \RR^d)} F_s(X, Z_s^\infty, q) (\nu_s^n - \nu_s^\infty)(\dif q) \dif s \\ 
    &\hspace{5em} + \int_t^T Z_s^\infty \cdot \big\{ B_s(X, Z_s^\infty, \mu^n) - B_s(X,Z_s^\infty, \mu^\infty) \big\} \dif s
\end{align*}
where, by assumption, it holds for some $C>0$ that
\[
\abs{F_t(X,Z_t^\infty,q)}  \leq C(1 + \abs{Z_t^\infty}^2 + \calW_1(\delta_0,q)) \quad \PP\text{-a.s.}
\]
Hence it follows from Proposition~\ref{prop:Portmanteau-thm-intYoung} and the dominated convergence theorem that
\begin{equation}\label{convergence-H}
    \int_t^T (H_s^n - H_s^\infty)(Z_s^\infty) \dif s \to 0
    \quad \PP\text{-a.s.\ for any}~t \in [0,T].
\end{equation}
\end{proof}

In view of the above corollary, to establish the existence of a fixed point it remains to identify a compact, convex subset of $\calP_1 \times \calY_1$ that the solution mapping maps to itself. In the following  section, we establish sufficient conditions for compactness in the space $\calY_1$. In particular, the range $\calQ^* \times \calK^*$ of our solution mapping introduced in Corollary \ref{cor:maps-onto-itself} turns out to be compact. 


\section{Compactness in the space of integrable Young measures}
\label{sec:cpt}

We start with the following basic lemma. The proof follows from the fact that weak relative compactness together with uniform integrability implies relative compactness in the $\calW_1$-topology (c.f.\ \cite[Corollary 5.6]{Carmona2018}).

\begin{lemma}\label{lem:general-W1-cpt}
    Let $\calQ \subset \calP_1(E)$ and $\frakK \subset \calP_1(\RR^d)$ be relatively compact subsets of probability measures. Then, the set
    \[
    \bar{\frakK} := \{q \in \calP_1(E \times \RR^d) \mid q^x \in \calQ,~ q^z \in \frakK \}
    \]
    is relatively compact in $\calP_1(E \times \RR^d)$.
\end{lemma}

The next lemma is essentially a corollary of \cite[Proposition~7.8]{carmona_lacker_2015}.

\begin{lemma}\label{lem:cpt-in-Banach}
For any reference measure $\frak{q}^x \in \bigcap_{p<\infty} \calP_p(E)$, and any $\delta, C > 0$, the set
\[
    \calQ := \left\{ q^x \in \calP_1(E) ~\middle|~ q^x \ll \frak{q}^x,~ \int_{E} \abs{\frac{\dif q^{x}}{\dif \frak{q}^x}}^{1+\delta} \dif \frak{q}^x \leq C \right\}
\]
is compact in $\calP_1(E)$ for any $\delta, C > 0$.
\end{lemma}
\begin{proof}
By Proposition 7.8 in \cite{carmona_lacker_2015}, the set $\calQ$ is weakly compact in $\calP(E)$.
Since for some $q > 1$,
\begin{align*}
    \sup_{q^x \in \calQ} \int_{E} \norm{x}_E^2 \dif q^x(x)
    &\leq \left(\int_{E} \norm{x}_E^{2q} \dif \frak{q}^x(x)\right)^{\frac1{q}}
    \sup_{q^x \in \calQ} \left( \int_{E} \abs{\frac{\dif q^x}{\dif \frak{q}^x}}^{1+\delta} \dif \frak{q}^x \right)^{\frac1{1+\delta}}
    < \infty,
\end{align*}
the set $\calQ$ is uniformly integrable and thus $\calW_1$-relatively compact.
Since  $\calW_1$-convergence implies  weak convergence and since $\calQ$ is weakly closed, it is also $\calW_1$-closed and hence $\calW_1$-compact.
\end{proof}

The next theorem proves a compactness condition for integrable Young measures. 

\begin{theorem}\label{thm:general-cpt}
Let $I$ be an index set and assume that the family of integrable measure-flows $\{q^i\}_{i \in I} \subset \calM_1$ satisfies, for some $\delta, C > 0$ and a reference measure $\frak{q}^x \in \bigcap_{p<\infty} \calP_p(E)$, that
\[
    q_t^{i,x} \ll \frak{q}^x \quad \text{for all}~i \in I,~ t \in [0,T],
\]
and
\[
    \sup_{i \in I} \int_{E} \abs{\frac{\dif q_t^{i,x}}{\dif \frak{q}^x}}^{1+\delta} \dif \frak{q}^x
    + \sup_{i \in I} \int_0^T \int_{\RR^d} \abs{w}^{1+\delta} q_t^{i,z}(\dif w) \dif t
    \leq C.
\]
Then, the family of integrable Young measures $\{\boldsym{\nu}^i\}_{i \in I} := \{\delta_{q_t^i}(\dif q) \dif t\}_{i \in I}$ is relatively compact w.r.t.~$\calS_1(\calP_1(E \times \RR^d))$ and so is their convex hull.
\end{theorem}
\begin{proof}
Due to Proposition~\ref{prop:tightness}, the family of Young measures $\{\boldsym{\nu}^i\}_{i \in I}$ is tight in the sense of Definition~\ref{def:tightness-Young} if there exists a compact set $\bar{\frakK}_\epsilon \subset \calP_1 (E \times \RR^d)$ for every $\epsilon > 0$ such that
\begin{equation}\label{tightness-of-frakK}
\sup_{i \in I} \lambda(\{t \in [0,T] \mid q_t^i \in \calP_1 (E \times \RR^d) \setminus \bar{\frakK}_\epsilon\}) \leq \epsilon.
\end{equation}
Moreover, by Proposition~\ref{prop:Prokhorov} tightness implies relative compactness w.r.t.~$\calS(\calP_1(E \times \RR^d))$. 

To find a compact set $\bar{\frakK}_\epsilon$ that satisfies condition \eqref{tightness-of-frakK}, we define a subset $\calQ \subset \calP_1(E)$ and parametrized subsets $\frakK(D) \subset \calP_1(\RR^d)$ for each $D>0$ by
\begin{align*}
    \calQ &:= \left\{ q^x \in \calP_1(E) ~\middle|~ q^x \ll \frak{q}^x,~ \int_{E} \abs{\frac{\dif q^{x}}{\dif \frak{q}^x}}^{1+\delta} \dif \frak{q}^x \leq C \right\} \\
    \frakK(D) &:= \left\{q^z \in \calP_1 (\RR^d) ~\middle|~ \int_{\RR^d} \abs{w} \mathbf{1}_{\{\abs{w} > K\}} \dif q^z (w) \leq \frac{D}{K^\delta} ~ \text{for all}~ K > 0 \right\},
\end{align*}
and let
\[
\bar{\frakK}(D) := \{ q \in \calP_1(E \times \RR^d) \mid q^x \in \calQ,~ q^z \in \frakK(D) \}.
\]

\textsl{Step 1.} We first prove that $\bar{\frakK}(D) \subset \calP_1(E \times \RR^d)$ is compact.
Relative compactness follows from Lemma~\ref{lem:general-W1-cpt} if we can prove that $\calQ$ and $\frakK(D)$ are relatively compact. By Lemma~\ref{lem:cpt-in-Banach}, the set $\calQ$ is relatively $\calW_1$-compact. Moreover, by definition, the set $\frakK(D)$ is tight and uniformly integrable, and hence also relatively $\calW_1$-compact. To see that $\bar{\frakK}(D)$ is closed, let us take a sequence $(q_n)_{n=1}^\infty \subset \bar{\frakK}(D)$ that converges to $q_\infty \in \calP_1 (E \times \RR^d)$. The compactness of $\calQ$ implies $q_\infty^x \in \calQ$ since $q_n^x \to q_\infty^x$ in $\calP_1(E)$. Moreover, by the Portmanteu theorem 
\[
\int_{\RR^d} \abs{w} \mathbf{1}_{\{\abs{w}>K\}} q_\infty^z (\dif w)
\leq \liminf_{n \to \infty} \int_{\RR^d} \abs{w} \mathbf{1}_{\{\abs{w} > K\}} q_n^z (\dif w)
\leq \frac{D}{\sqrt{K}},
\]
because the indicator function $\mathbf{1}_{\{\abs{w} > K\}}$ is lower semicontinuous. Thus, $q_\infty$ belongs to $\bar{\frakK}(D)$. 

\mbox{}

\textsl{Step 2.} Now we prove the tightness of $\{\boldsym{\nu}^i\}_{i \in I}$. To this end, we fix $\epsilon > 0$ and find a constant $D_\epsilon > 0$ that satisfies the inequality \eqref{tightness-of-frakK} with $\bar{\frakK}_\epsilon = \bar{\frakK}(D_\epsilon)$. For this, we introduce the measurable subsets 
\[
E_{i,D} := \{t \in [0,T] \mid q^i_t \notin  \bar{\frakK}(D)\} \subset [0,T]
\]
in terms of which the inequality \eqref{tightness-of-frakK} can be rewritten as 
$$\sup_{i \in I} \lambda (E_{i,D}) < \epsilon.$$
If $q \in \calP_1(\RR^d)$ does not belong to $\frakK(D)$, then there exists a constant $K > 0$ such that
\begin{align*}\label{q-notin-frakK}
    \int_{\RR^d} \abs{w}^{1+\delta} \dif q(w)
    \geq K^\delta \int_{\RR^d} \abs{w} \mathbf{1}_{\{\abs{w}>K\}} \dif q(w) > D
\end{align*}
and so
\begin{equation}\label{chara-set-EiD}
\int_{\RR^d} \abs{w}^{1+\delta} \dif q_t^{i,z} (w) > D
\quad \text{for all}~t \in E_{i,D}.
\end{equation}
If for any $D > 0$ there would exist $i(D) \in I$  such that $\lambda(E_{{i(D)},D}) > \epsilon$, then for any $D > 0$ 
\begin{align*}
C &\geq \int_0^T \int_{\RR^d} \abs{w}^{1+\delta} q_t^{i(D),z} (\dif w) \dif t 
\geq \int_{E_{{i(D)},D}} \int_{\RR^d} \abs{w}^{1+\delta}  q_t^{i(D),z} (\dif w) \dif t
> \epsilon D,
\end{align*}
which is impossible. Thus, there exists $D_\epsilon > 0$ for any $\epsilon>0$ that satisfies \eqref{tightness-of-frakK} with $\bar{\frakK}_\epsilon = \bar{\frakK}(D_\epsilon)$. 

\mbox{}

\textsl{Step 3.}
By Proposition~\ref{prop:cpt-ible-Young-meas}, it remains to verify uniform integrability of the family $\{\boldsym{\nu}^i\}_{i \in I}$.
This follows from 
\begin{align*}
\int_0^T \calW_1 (\delta_0, q^i_t) &\mathbf{1}_{\{\calW_1 (\delta_0, q^i_t) > K\}} \dif t \\
&\leq \frac1{K^{\delta}} \int_0^T \calW_1 (\delta_0, q_t^i)^{1+\delta} \dif t  \\
&\leq \frac{2^\delta}{K^{\delta}} \left[\int_0^T \int_{E} \norm{x}_E^{1+\delta} q_t^{i,x} (\dif x) \dif t + \int_0^T \int_{\RR^d} \abs{w}^{1+\delta} q_t^{i,z} (\dif w) \dif t\right] \\
&\leq \frac{2^\delta}{K^{\delta}} \left[ T M_p(\frak{q}^x) C^{\frac1{1+\delta}} + C \right],
\end{align*}
where $p>0$ in the last line is taken large enough.
\end{proof}

The following corollary considers the important special case where our integrable Young measures are defined in terms of probability measures with BMO-bounded densities.

\begin{corollary}\label{cor:cpt-laws}
Suppose we are given a random variable $X$, processes $Z^n \in \HH^2_{\operatorname{BMO}}(\PP)$, measurable functions $\theta^n$, and probability measure $\PP^i \in \calP(\Omega)$ satisfying the assumptions of Lemma~\ref{lem:uni-bdd-norm}.
Then, the families of probability laws and integrable Young measures,
\[
    \{ \calL^i(X) \}_{i \in I}, \quad
    \big\{ \delta_{\calL^i (X,Z_t^i)} (\dif q) \dif t  \big\}_{i \in I},
\]
are relatively compact w.r.t.~the $\calW_1$-topology and $\calS_1(\calP_1(E \times \RR^d))$, respectively.
\end{corollary}
\begin{proof}
The result follows from Lemma~\ref{lem:cpt-in-Banach} and Theorem~\ref{thm:general-cpt} applied to
\[
\frak{q}^x := \PP \circ X^{-1}, \quad q_t^{i,x} :\equiv \PP^i \circ X^{-1}, \quad
q_t^{i,z} := \PP^i \circ Z_t^{-1}.
\]
In fact, $q_t^{i,x} \ll \frak{q}^x$ because $\PP^i \ll \PP$.
Moreover, the assumptions of Theorem~\ref{thm:general-cpt} are satisfied by \eqref{uni-bdd-density} and \eqref{uni-bdd-norm} in Lemma~\ref{lem:uni-bdd-norm}.
\end{proof}


\section{Existence of solutions to generalized MV-BSDE}
\label{sec:existence}

In this section, we prove our main results, Theorem~\ref{thm:existence-with-bdd} and Theorem \ref{thm:relax-bdd}, and hence the existence of MFG equilibria in weak formulation.

\subsection{Proof of Theorem~\ref{thm:existence-with-bdd}}
\label{subsec:existence-bdd}



By Corollary~\ref{cor:maps-onto-itself}, the solution mapping $\Phi$ maps the convex, closed set $\calQ^\ast \times \calK^\ast$ to itself.
In view of Lemma~\ref{lem:cpt-in-Banach}, the convex set $\calQ^\ast \subset \calP_1$ is compact. Corollary~\ref{cor:cpt-laws} implies that the convex set $\calK^\ast \subset \calY_1$ is also compact. It follows from Corollary~\ref{cor:FP-continuity} that $\Phi$ is continuous on $\calQ^\ast \times \calK^\ast$. Moreover, as shown in Appendix~\ref{appx:embed-Young-meas} the space $\calP_1 \times \calY_1$ can be embedded into a locally convex Hausdorff topological space. Therefore, Schauder's fixed-point theorem is applicable to the solution mapping and yields a fixed point
\[
(\mu, \boldsym{\nu})
= (\calL^{\mu, \boldsym{\nu}}(X),
  \delta_{\calL^{\mu, \boldsym{\nu}}(X, Z_t^{\mu, \boldsym{\nu}})}(\dif q) \dif t) \in \calP_1 \times \calY_1.
\]
In particular, $\calL^{\mu,\boldsym{\nu}}(X, Z_\cdot^{\mu,\boldsym{\nu}}) \in \calM_1$. Since the disintegration $\nu$ of $\boldsym{\nu}$ is unique, 
\[
\nu_t = \delta_{\calL^{\mu, \boldsym{\nu}} (X, Z_t^{\mu, \boldsym{\nu}})} \quad 
    \mbox{for a.e. $t \in [0,T]$},
\]
and the BSDE \eqref{BSDE-mu-nu} turns into
\begin{align*}
\dif Y_t^{\mu, \boldnu} &= -H_t (X, Z_t^{\mu,\boldnu}, \calL^{\mu,\boldnu}(X,Z_s^{\mu,\boldnu})) \dif t + Z_t^{\mu,\boldnu} \dif W_t, 
\quad Y_T^{\mu,\boldnu} = G(X, \calL^{\mu,\boldnu}(X)), \\
\frac{\dif \PP^{\mu, \boldsym{\nu}}}{\dif \PP}
&= \calE\left(
   B_\cdot (X, Z_\cdot^{\mu, \boldsym{\nu}}, \calL^{\mu, \boldsym{\nu}}(X)) \cdot W
\right)_T.
\end{align*}
Hence, $(X, Y^{\mu, \boldsym{\nu}}, Z^{\mu, \boldsym{\nu}},\PP^{\mu,\boldsym{\nu}})$ is a solution to the generalized MV-BSDE \eqref{general-MVBSDE}.


\subsection{MV-BSDEs with unbounded parameters}
\label{subsec:existence-unbdd}

We now extend the previous result to the case of unbounded parameters. 
The following is a modification of Theorem 3.8 in \cite{hao2025mean}. 

\begin{proposition}\label{prop:new-apriori-bound}
Under Assumption~\ref{asm:strong-quad}, any solution $(X,Y,Z,\barPP) \in \SS^2 (\barPP) \times \SS^\infty (\barPP) \times \HH^2_{{\operatorname{BMO}}} (\barPP) \times \calP(\Omega)$ to the generalized McKean-Vlasov FBSDE \eqref{general-MVFBSDE} is bounded: 
\begin{equation}\label{eq:new-apriori-XYZ}
\norm{X}_{\SS^2(\bar{\PP})} \leq \bar{L}_x, \quad
\norm{Y}_{\SS^\infty(\bar{\PP})} \leq \bar{L}_y, \quad
\norm{Z}_{\HH^2_{{\operatorname{BMO}}} (\barPP)} \leq \bar{L}_z
\end{equation}
for some constants $\bar{L}_x, \bar{L}_y, \bar{L}_z$ that depend only on the constants $L, K_\sigma, K, \gamma, \tilde{\gamma}$ given in Assumptions~\ref{asm:BSDE-standing-asm} and \ref{asm:strong-quad}.
\end{proposition}
\begin{proof}
If Condition (1a) of Assumption~\ref{asm:strong-quad} is satisfied, the result follows from Theorem 3.8 in \cite{hao2025mean}. In what follows we outline the key arguments under Condition (1b).

\mbox{}

\textsl{Step 1.}
The dynamics of $X$ under the probability measure $\barPP$ is given by
\[
\dif X_t
= \sigma_t (X) B_t (X, Z_t, \barcalL(X)) \dif t
  + \sigma_t (X) \dif \bar{W}_t.
\]
We set $X_t^\ast := \sup_{s \in [0,t]} \abs{X_s}$.
From condition (1b) of Assumption~\ref{asm:strong-quad}, we have that 
\begin{align*}
\bar{\EE}[ \abs{X_s^\ast}^2]
&\leq 3 \left[
  \abs{x_0}^2
  + 3T \gamma^2 \int_0^t \left(
    1 + \bar{\EE}[ \abs{X_s^\ast}^2 ]
    + \bar{\EE}[\abs{Z_s}^2]
  \right) \dif s
  + 4 \bar{\EE}\left[\int_0^t \abs{\sigma_s (X)}^2 \dif s \right] \right] \\
&\leq C \left( 1 + \int_0^t \bar{\EE}[ \abs{X_s^\ast}^2 ] \dif s + \bar{\EE} \left[ \int_0^t \abs{Z_s}^2 \dif s \right] \right)
\end{align*}
for some $C>0$, and an application of Gronwall's inequality yields for some $L_1>0$
\begin{equation}\label{control-X-by-Z}
\bar{\EE}[ \abs{X_T^\ast}^2 ]
\leq L_1 \left(1
  + \bar{\EE}\left[\int_0^T \abs{Z_s}^2 \dif s \right] \right).
\end{equation}

\mbox{ } 

\textsl{Step 2.}
Analogously to the proof of Theorem 3.8 in \cite{hao2025mean}, we prove that the norms of the solutions are controlled by the expectation of $X$ as
\begin{equation}\label{bound-by-EXP-X}
\begin{aligned}
\norm{Y}_{\SS^\infty_{[t,T]} (\barPP)}
\leq L_2 \left( 1 + \bar{\EE}\left[ \abs{X_T^\ast} \right] \right), \quad
\bar{\EE} \left[\int_0^T \abs{Z_s}^2 \dif s \right]
\leq L_2 \left( 1 + \bar{\EE}\left[ \abs{X_T^\ast} \right] \right),
\end{aligned}
\end{equation}
for some positive constants $L_2$. Let
\[
\Psi_t (y)
:= \exp\left\{
  4K y
  + 4K \int_0^t 
    \left( 2K + \gamma \left( 1 + \bar{\EE}\left[ \abs{X_s^\ast} \right]
    + \bar{\EE}[\abs{Z_s}] \right) \right) \dif s
\right\}.
\]
An application of the It\^o-Tanaka formula to $\Psi_t(\abs{Y}_t)$ yields for some $C>0$ \footnote{We refer to \cite{hao2025mean} for the detailed computation.}
\begin{equation}\label{bound-Y-by-XYZ}
\begin{split}
\norm{Y}_{\SS^\infty_{[t,T]} (\barPP)}
\leq C \left( 1 + \bar{\EE}\left[ \abs{X_T^\ast} \right]  \right) + \gamma \int_t^T \bar{\EE}[\abs{Z_s}] \dif s.
\end{split}
\end{equation}

The strictly quadratic growth condition allows us to eliminate the integral of the $Z$-process from the above equation. In fact, by condition \eqref{cond:strong-quad} in Assumption~\ref{asm:strong-quad},
\begin{equation}\label{Z-without-EXP}
\begin{split}
\tilde{\gamma}  &\int_{t}^{T} \abs{Z_s}^2 \dif s \leq Y_T - Y_t
  + \int_{t}^{T} \gamma \left( 1 + \bar{\EE}\left[ \abs{X_s^\ast} \right] + \bar{\EE}[\abs{Z_s}] \right) \dif s - \int_t^T Z_s \dif \bar{W}_s.
\end{split}
\end{equation}
By Young's inequality, 
\begin{equation}\label{Young-ineq-EZ}
\gamma \int_t^T \bar{\EE}[\abs{Z_s}] \dif s \leq \frac{\epsilon \tilde{\gamma}}2 \int_t^T \bar{\EE}[\abs{Z_s}^2] \dif s + \frac{T \gamma^2}{2 \epsilon \tilde{\gamma}}
\end{equation}
for any $\epsilon > 0$. Taking expectations in \eqref{Z-without-EXP} and applying the above inequality with $\epsilon = 1$ yields 
\begin{equation}\label{bound-Z-by-XY}
\begin{split}
\frac{\tilde{\gamma}}2 \bar{\EE}\left[ \int_t^T \abs{Z_s}^2 \dif s \right]
&\leq 2\norm{Y}_{\SS^\infty_{[t,T]} (\barPP)} + \frac{T \gamma^2}{2 \tilde{\gamma}} + T \gamma \left( 1 +  \bar{\EE}\left[ \abs{X_T^\ast} \right] \right) .
\end{split}
\end{equation}

Therefore, the first estimate in \eqref{bound-by-EXP-X} follows by plugging \eqref{bound-Z-by-XY} into \eqref{Young-ineq-EZ} and then \eqref{Young-ineq-EZ} into \eqref{bound-Y-by-XYZ} with sufficiently small $\epsilon$.
The second inequality follows from the first one and \eqref{bound-Z-by-XY}. 

\mbox{ } 

\textsl{Step 3.} We now establish the desired bounds on $(X,Y,Z)$.
Plugging \eqref{bound-by-EXP-X} into \eqref{control-X-by-Z}, 
\begin{align*}
\bar{\EE}[\abs{X_T^*}^2]
&\leq L_1 \big( 1 + L_2 (1 + \bar{\EE}[\abs{X_T^\ast}])  \big)
\le C + \frac12 \bar{\EE}[\abs{X_T^*}^2]
\end{align*}
for some $C>0$.
This result yields the a priori bound for $X$ in \eqref{eq:new-apriori-XYZ}. 
The estimate for $Y$ in \eqref{eq:new-apriori-XYZ} follows from that for $X$ and \eqref{bound-by-EXP-X}.
Taking the conditional expectation of \eqref{Z-without-EXP}, we get
\begin{align*}
\tilde{\gamma} &\bar{\EE}\left[ \int_\tau^T \abs{Z_s}^2 \dif s ~\middle|~ \calF_\tau \right] 
\leq 2 \norm{Y}_{\SS^\infty_{[t,T]} (\barPP)} + \frac{T \gamma^2}{2 \tilde{\gamma}}
+ T \gamma \left( 1 +  \bar{\EE}\left[ \abs{X_T^\ast} \right] \right)
+ \frac{\tilde{\gamma}}2 \int_0^T \bar{\EE}[\abs{Z_s}^2] \dif s
\end{align*}
for any $\FF$-stopping time $\tau \ge 0$.
Plugging \eqref{bound-by-EXP-X} into the above inequality, we conclude that
\begin{align*}
&\bar{\EE}\left[ \int_\tau^T \abs{Z_s}^2 \dif s ~\middle|~ \calF_\tau \right]
\leq C (1 + \bar{L}_x + \bar{L}_y).
\end{align*}
\end{proof}

The preceding result allows us to establish the desired existence result using a truncation argument. To this end, we choose continuous cut-off functions $c_N = (c_N^x, c_N^z)$ that satisfy
\begin{equation}\label{cutoff-functions}
    \norm{c_N(x,z) - (x,z)} \leq \norm{(x,z)} \boldsym{1}_{\{\norm{(x,z)} \geq N\}} \quad
    \mbox{and} \quad
    \norm{c_N(x,z)} \leq N \wedge \norm{(x,z)},
\end{equation}
and define the corresponding truncated functions $G^N, B^N, F^N$, and $H^N$ by
\begin{equation*}\label{N-truncated-functions}
\begin{split}
    G^N (x,q^x) & := G(x, q^x \circ (c_N^x)^{-1}), \quad 
    B_t^N (x,z,q^x) := B_t(x, z, q^x \circ (c_N^x)^{-1}), \\
    F_t^N (x,z,q) & := F_t(x,z,q \circ c_N^{-1}), \quad H_t^N(x,y,z,q) := F_t^N(x,y,z,q) + B_t^N(x,z,q^x) \cdot z
\end{split}
\end{equation*}
By Theorem \ref{thm:existence-with-bdd} the respective MV-BSDEs admit solutions 
\begin{equation} \label{N-solution}
  (X,Y^N,Z^N,\PP^N) \in \SS^2(\PP) \times \SS^\infty(\PP) \times \HH_{\operatorname{BMO}}^2(\PP) \times \calP(\Omega).  
\end{equation}

The next lemma provides an auxiliary convergence result for the cut-off functions from which we subsequently conclude that the solutions to the ``truncated'' MV-BSDEs converge to a solution of the original quadratic MV-BSDE. The assertion follows from \eqref{cutoff-functions}.

\begin{lemma}\label{lem:preliminary-convergence-cN}
    Let $t \mapsto q_t^N$ be  $\calP_2(\calC_d \times \RR^d)$-valued Borel measurable maps on $[0,T]$. If
    \begin{equation*}
        \sup_{N \in \NN} \int_0^T \int_{\calC_d \times \RR^d} \norm{(x,z)}^2 \dif q_t^N (x,z) \dif t < \infty,
    \end{equation*}
    then any continuous cut-off function $c_N : \calC_d \times \RR^d \to \calC_d \times \RR^d$ that satisfies \eqref{cutoff-functions} also satisfies 
    \begin{equation*}
        \int_0^T \calW_1 \big( q_t^N \circ c_N^{-1}, q_t^N \big) \dif t \to 0. 
    \end{equation*}
\end{lemma}

\bigskip

\textsc{Proof of Theorem~\ref{thm:relax-bdd}}.
If $G, B$ and $F$ satisfy Assumption~\ref{asm:strong-quad} for some constants $\gamma, \gamma_F, \tilde{\gamma}_F$, then so do $G^N, B^N $ and $F^N$. Hence, Proposition~\ref{prop:new-apriori-bound} yields bounds $\bar{L}_x, \bar{L}_z > 0$ s.t.
\begin{equation}\label{eq:uniform-bound-XZN}
\sup_{N \in \NN} \norm{X}_{\SS^2(\PP^N)} \leq \bar{L}_x, \quad
\sup_{N \in \NN}\norm{Z^N}_{\HH^2_{{\operatorname{BMO}}} (\PP^N)} \leq \bar{L}_z 
\end{equation}
and hence 
\begin{equation}\label{uniform-bound-BN}
    \sup_{N \in \NN} \abs{B_\cdot^N(X, 0, \calL^N(X))}^2
    \leq C ( 1  + \bar{L}_x) < \infty.
\end{equation}
In particular, $\sup_{N \in \NN} \norm{B_\cdot^N(X, Z_\cdot^N, \calL^N(X))}_{\HH_{\operatorname{BMO}}^2(\PP^N)} < \infty$. Thus, by Proposition~\ref{prop:equiv-sol-P-barP},
\begin{equation*}
\begin{aligned}
    \sup_{N \in \NN} \norm{Z^N}_{\HH_{\operatorname{BMO}}^2(\PP)} < \infty.
\end{aligned}
\end{equation*}


Let $\mu^N := \calL^N(X), \quad q^N_t := \calL^N(X,Z_t^N), \quad \boldsym{\nu}^N (\dif q, \dif t) := \delta_{q_t^N}(\dif q) \dif t.$
By Corollary~\ref{cor:cpt-laws} we may without loss of generality assume that 
\begin{equation}\label{barnu-to-limit}
    \mu^N \to \mu^* \in \calP_1 \quad \mbox{and} \quad \boldsym{\nu}^N \to \boldsym{\nu}^* \in \calY_1.
\end{equation}
Furthermore, setting $\bar{q}_t^N := q_t^N \circ c_N^{-1}$ Lemma~\ref{lem:preliminary-convergence-cN} yields 
\[
\bar{\mu}^N := \mu^N \circ (c_N^x)^{-1} \to \mu^* \quad \mbox{and} \quad 
\bar{\boldnu}^N (\dif q, \dif t) := \delta_{\bar{q}_t^N}(\dif q) \dif t \to \boldnu^\ast.
\]
From the definition \eqref{def-sol-map} of the solution map, we obtain 
\[
    \Phi(\bar{\mu}^N, \bar{\boldnu}^N) = (\mu^N, \boldnu^N) \to (\mu^\ast, \boldnu^\ast)
    \quad \text{in} \quad \calP_1 \times \calY_1
\]
and from Corollary~\ref{cor:FP-continuity} and \eqref{barnu-to-limit} we obtain 
\[
    \Phi(\bar{\mu}^N, \bar{\boldnu}^N) \to \Phi(\mu^\ast, \boldnu^\ast)
    \quad \text{in} \quad \calP_1 \times \calY_1.
\]

\appendix
 
\section{BSDEs with stochastic Lipschitz drivers}

In this appendix, we recall/establish an auxiliary comparison principle for a family of BSDEs with stochastic Lipschitz drivers. The result is essentially a corollary \cite[Theorem 10]{briand2008bsdes}. Specifically, for some index set $I$ and any $i \in I$ we consider the BSDEs
\[
    dY^i_t = H^i_t(\omega,Y^i_t,Z^i_t)dt + Z^i_t dW_t, \quad Y^i_T = \xi^i
\]
defined on $(\Omega, \calF, (\calF_t), \PP)$. We assume that the drivers $H_t^i(\omega,y,z)$ are $\FF$-progressively measurable for all $(y,z) \in \RR \times \RR^d$ and that the terminal conditions $\xi^i$ are $\calF_T$-measurable. 

\begin{definition}[Comparison principle]\label{def:comparison-principle}
We say that the above family of BSDEs satisfies a comparison principle if, for any two indices $i,j \in I$ with 
\begin{align*}
    \xi^i \geq \xi^j \quad \PP\text{-a.s.} \quad \mbox{and} \quad
    H_t^i(y,z) \geq H_t^j(y,z) \quad \dif t \times \dif \PP\text{-a.s. for all}~(y,z) \in \RR \times \RR^d 
\end{align*}
the corresponding solutions $(Y^i, Z^i), (Y^j,Z^j) \in \SS^2(\RR;\PP) \times \HH^2(\RR^d;\PP)$ satisfy  
\begin{align*}
    Y_t^i \geq Y_t^j \quad \text{for all}~t \in [0,T],~\PP\text{-a.s.}
\end{align*}
\end{definition}

\begin{proposition}\label{prop:comparison-thm}
Suppose that $\xi^i$ and $H^i$ satisfy the following additional conditions $\PP$-a.s.~for any $i \in I$:
\begin{enumerate}
    \item \label{cond:stoc-Lip-H} 
    For some $\alpha \in (0,1)$ and $K \in \HH^2_{\operatorname{BMO}}(\PP)$ such that $K \geq 1$ and $\EE \left[ \calE\left(K \cdot W\right)_T^{q_\ast} \right] < \infty$ for some $q_\ast > 1$, the driver $H^i$ satisfies
    \begin{align*}
        (y-\bar{y})(H_t^i(y,z) - H_t^i(\bar{y},z)) &\leq K_t^{2\alpha}\abs{y-\bar{y}}^2, \\
        \abs{H_t^i(y,z) - H_t^i(y,\bar{z})} &\leq K_t \abs{z-\bar{z}}, \quad
        \text{for all}~y, \bar{y} \in \RR, \, z,\bar{z} \in \RR^d.
    \end{align*}
    \item \label{cond:p-integrability-H}
    The following integrability condition holds for some constants (that may depend on $i \in I$) $p^\ast > p_\ast$ where $p_\ast > 1$ is the H\"older conjugate of $q_\ast$:
    \[
    \EE \left[ \abs{\xi^i}^{p^\ast} + \left( \int_0^T \abs{H_s^i(0,0)} \dif s \right)^{p^\ast} \right] < \infty
    \]
\end{enumerate}
Then, the BSDEs with drivers $H^i$ and terminal conditions $\xi^i$ admit unique solutions
\[
    (Y^i, Z^i) \in \bigcap_{p < p^\ast} (\SS^p(\PP) \times \HH^p(\PP))
\]    
and satisfy the comparison principle.
\end{proposition}
\begin{proof}
The existence and uniqueness of solutions follows from Theorem 10 in \cite{briand2008bsdes}. The comparison principle is standard. Let $i,j \in I$ with
\[
\xi^i \geq \xi^j \quad \PP\text{-a.s.} \quad
H_t^i(y,z) \geq H_t^j(y,z) \quad \dif t \times \dif \PP\text{-a.s. for all}~(y,z) \in \RR \times \RR^d.
\]
Let
\[
    \Delta Y_t = Y_t^i - Y_t^j, \quad
    \Delta Z_t = Z_t^i - Z_t^j, \quad
    \Delta H_t (y,z) = H_t^i(y,z) - H_t^j(y,z) 
\]
and set 
\begin{align*}
\delta_z H_t^i &:= \frac{H_t^i (Y_t^i, Z_t^i) - H_t^i (Y_t^j, Z_t^j)}{\abs{\Delta Y_t}} \mathbf{1}_{\{\abs{\Delta Y_t} \neq 0\}}, \\
\delta_z H_t^i & := \frac{H_t^i (Y_t^j, Z_t^i) - H_t^i (Y_t^j, Z_t^j)}{\abs{\Delta Z_t}^2} \mathbf{1}_{\{\abs{\Delta Z_t} \neq 0\}} \Delta Z_t.
\end{align*}
Then, $(\Delta Y, \Delta Z) \in \SS^p(\PP) \times \HH^p(\PP)$ satisfies the linear BSDE
\[
\Delta Y_t = \int_t^T (\delta_y H_s^i \Delta Y_s + \delta_z H_s^i \Delta Z_s + \Delta H_s (Y_s^j, Z_s^j)) \dif s + \int_t^T \Delta Z_s \dif W_s.
\]
Following a similar argument as in the proof of  \cite[Lemma 7]{briand2008bsdes}, we obtain that
\[
e_t \Delta Y_t = \EE^{\PP^\ast}\left[ e_T \Delta \xi + \int_t^T e_s \Delta H_s(Y_s^j, Z_s^j) \dif s ~\middle|~ \calF_t \right],
\]
where $e_t := \exp\big\{\int_0^t \delta_y H_s^i \dif s \big\}$ and $\PP^\ast$ is equivalent to $\PP$ with density 
\[
    \frac{\dif \PP^\ast}{\dif \PP} = \calE\left( \delta_z H^i \cdot W \right)_T.
\]
The assertion now follows from standard arguments.
\end{proof}


\section{BMO martingales}\label{appx:BMO}

This appendix recalls key properties of BMO martingales.\footnote{The reader is referred to \cite{kazamaki2006BMO} for a detailed discussion of BMO martingales.} We continue working on our probability space $(\Omega,\calF,\PP)$ and denote by $\mathcal{T}$ the set of all $[0,T]$-valued $\FF$-stopping times.

\begin{definition}\label{def:BMO-space}
The BMO norm $\norm{H}_{\HH^2_{\operatorname{BMO}} (\RR^d; \PP)}$ of a process $H \in \HH^2 (\RR^d; \PP)$ is defined by 
\[
    \norm{H}_{\HH^2_{\operatorname{BMO}} (\RR^d; \PP)} := \sup_{\tau \in \mathcal{T}} \norm{\EE^\PP \left[ \int_\tau^T \abs{H_s}^2 \dif s \; \middle| \; \calF_\tau \right]}_{L^\infty(\Omega;\RR;\PP)}.
\]
We write $\HH^2_{\operatorname{BMO}} (\RR^d; \PP)$ for the set of all processes with finite BMO norm under $\PP$.
\end{definition}

\begin{proposition}(\cite[Theorem 3.3 and 3.6]{kazamaki2006BMO})\label{prop:BMO-norm-equivalence}
For any  $M>0$, there exist two constants $\gamma_1, \gamma_2$ that only depend on $M$ such that, for any two processes $\theta, H \in \HH^2_{{\operatorname{BMO}}} (\PP)$ such that $\|\theta\|_{\HH^2_{{\operatorname{BMO}}} (\PP)} \leq M$ the process $H$ belongs to  $\HH^2_{{\operatorname{BMO}}} (\PP^\theta)$ as well and satisfies
\[
\gamma_1\|H\|_{\HH^2_{{\operatorname{BMO}}} (\PP)} \leq \|H\|_{\HH^2_{{\operatorname{BMO}}} (\PP^\theta)} \leq \gamma_2\|H\|_{\HH^2_{{\operatorname{BMO}}} (\PP)} \quad \mbox{where} \quad 
\frac{\dif \PP^\theta}{\dif \PP} := \calE(\theta \cdot W)_T.
\]
\end{proposition}

\begin{proposition}[Energy inequality]\label{prop:energy-ineq}
Every $H \in \HH^2_{{\operatorname{BMO}}} (\PP)$ satisfies the following inequality:
\[
\EE \left[ \left( \int_0^T \abs{H_s}^2 \dif s \right)^n \right] \leq n! \|H\|^{2n}_{\HH^2_{{\operatorname{BMO}}} (\PP)} \quad \text{for all}~n \in \NN.
\]
\end{proposition}

\begin{proposition}(\cite[Lemma A.2]{herdegen2021equilibrium} and \cite[Theorem 3.1]{kazamaki2006BMO})\label{prop:BMO-chara}
For any $M>0$, there exist some $p \in (1,\infty)$ and $C_p > 0$ that only depends on $p$ such that every process $H \in \HH^2_{{\operatorname{BMO}}} (\PP)$ with $\norm{H}_{\HH^2_{{\operatorname{BMO}}} (\PP)} \leq M$ satisfies the following for all $\tau \in \mathcal{T}$:
\begin{enumerate}
    \item The following inequality holds:
    \[
    \norm{\EE_\tau \left[\left(\frac{\calE(H \cdot W)_\tau}{\calE(H \cdot W)_T}\right)^{\frac1{p-1}}\right]}_{L^\infty (\PP)} \leq C_p,
    \]
    \item The reverse H\"older inequality holds:
    \[
    \EE_\tau^{\PP} [\calE(H \cdot W)_T^p] \leq C_p \calE(H \cdot W)_\tau^p.
    \]
\end{enumerate}
\end{proposition}


\section{Young measures and stable topologies}\label{appx:Young-measure}

This appendix recalls key properties of (integrable) Young measures and the stable topology that we utilize. A detailed discussion of Young measures can be found in \cite{castaing2004young, florescu2012young}.

\subsection{Young measures}
Throughout, $(\Omega, \mathcal{A}, \mu)$ is a positive measure space, $(E, \tau_E)$ is a Polish space with metic $d$, and $E$ and $\calP(E)$ are equipped with their respective Borel $\sigma$-algebra.

\begin{definition}[Young measures]\label{def:Young-meas}
If a positive measure $\boldsym{\nu}$ on $\Omega \times E$ satisfies
\[
\boldsym{\nu}(A \times E) = \mu(A) \quad
\text{for all}~A \in \mathcal{A},
\]
then $\boldsym{\nu}$ is called a Young measure on $\Omega \times E$.
We denote by $\calY(\Omega \times E)$, or $\calY (E)$ in short, the space of Young measures on $\Omega \times E$.
\end{definition}


\begin{proposition}(\cite[Theorem 3.2 and Remark 3.3]{florescu2012young} or \cite[p.19-20]{castaing2004young})\label{prop:disintegration}
For every Young measure $\boldsym{\nu} \in \calY(\Omega \times E)$, there exists a measurable map $\nu_\cdot: \Omega \to \calP(E)$ such that
\[
\boldsym{\nu}(A \times B) = \int_A \nu_\omega (B) \dif \mu(\omega)
= \int_\Omega \int_{\calP(E)} \mathbf{1}_{A \times B} (\omega, q) \nu_\omega (\dif q) \mu (\dif \omega)
\]
for any measurable set $A \subset \Omega$ and $B \subset \calP(E)$.
The mapping $\nu_\cdot$ is called \emph{disintegration} of the measure $\boldsym{\nu}$ w.r.t.~$\mu$. The disintegration is unique in the sense that any two disintegrations $\nu_\cdot, \nu'_\cdot$ of $\boldsym{\nu}$, satisfy $\nu_\cdot = \nu'_\cdot$ $\mu$-almost everywhere.
\end{proposition}


\begin{definition}[Stable topology]\label{def:stable-topology}
The stable topology on $\calY(E)$ is the weakest topology for which the map
\begin{equation}\label{eq:linear-functional-on-calM}
\begin{cases}
  \calY & \to \RR, \\
  \boldsym{\nu} & \mapsto \int_{\Omega \times E} \mathbf{1}_A (\omega) f(q) \dif \boldsym{\nu} (\omega, q)
\end{cases}
\end{equation}
is continuous for any $A \in \mathcal{A}$ and any bounded, Lipschitz continuous function $f:E \to \RR$.\footnote{This definition can be justified by Proposition 3.22 in \cite{florescu2012young}.} The stable topology on $\calY(E)$ is denoted by $\calS(E)$.
\end{definition}

\begin{proposition}(\cite[Proposition 3.25]{florescu2012young})\label{prop:spaceY-metrizable}
If $\mathcal{A}$ is countably generated, then the topology $\calS(E)$ is metrizable.
\end{proposition}


\begin{definition}[Tightness]\label{def:tightness-Young}
A subset $\calK \subset \calY(E)$ is called \emph{tight} if, for any positive value $\epsilon > 0$, there exists a compact set $K \subset S$ such that
\[
\sup_{\boldsym{\nu} \in \calK} \boldsym{\nu}(\Omega \times (S \setminus K)) < \epsilon.
\]
\end{definition}

\begin{proposition}(\cite[Proposition 3.54]{florescu2012young})\label{prop:tightness}
Let $\calU$ be a family of measurable mappings $u:\Omega \to E$, and let 
\[
    \boldsym{\nu}^u (\dif \omega, \dif q) := \delta_{u_\omega} (\dif q) \mu(\dif \omega) \quad \mbox{for} \quad u \in \calU.
\]
Then, the family of Young measures $(\boldsym{\nu}^u)_{u \in H}$ is tight if and only if, for any $\epsilon > 0$, there exists a compact subset $K \subset E$ such that
\[
\sup_{u \in H} \mu(u^{-1} (E \setminus K)) < \epsilon.
\]
\end{proposition}

\begin{proposition}(The Prokhorov theorem; \cite[Theorem 3.59]{florescu2012young})\label{prop:Prokhorov}
A subset $\mathcal{H} \subset \calY(E)$ is relatively compact w.r.t.~the stable topology $\calS(E)$ if and only if $\mathcal{H}$ is tight in $\calY(E)$.
\end{proposition}

\subsection{Integrable Young measures}

\begin{definition}[Integrable Young measures]\label{def:integrable-Young-meas}
An element $\boldsym{\nu} \in \calY (\Omega \times E)$ is called an integrable Young measure if for some $q_0 \in E$ 
\[
\int_{\Omega \times E} d(q_0,q) \dif \boldsym{\nu}(\omega,q) < \infty.
\]
We denote the class of integrable Young measures on $\Omega \times E$ by $\calY_1 (\Omega \times E)$, or $\calY_1 (E)$ for short. 
\end{definition}

\begin{definition}[Stable topology on $\calY_1(E)$]
The stable topology $\calS_1(E)$ on $\calY_1(E)$ is the weakest topology for which the mapping \eqref{eq:linear-functional-on-calM} is continuous for any $A \in \mathcal{A}$ and any Lipschitz continuous function $f:E \to \RR$.\footnote{This definition can be justified by Proposition 2.4.1 in \cite{castaing2004young}.}
\end{definition}


\begin{proposition}(\cite[Proposition 2.4.1]{castaing2004young})\label{prop:Portmanteau-thm-intYoung}
Let $\{\boldnu^i\}_{i \in I}$ be a net in $\calY_1(E)$ and $\boldsym{\nu}^\infty \in \calY_1(E)$. The following conditions are equivalent:
\begin{enumerate}[(i)]
    \item $\boldnu^i \to \boldsym{\nu}^\infty$ in $\calS_1(E)$,
    \item $\boldnu^i \to \boldsym{\nu}^\infty$ in $\calS(E)$ and $\int_{\Omega \times E} d(q_0,q) \dif \boldnu^i (\omega,q) \to \int_{\Omega \times E} d(q_0,q) \dif \boldsym{\nu}^\infty (\omega,q)$ for some $q_0 \in S$,
    \item $\int_{\Omega \times E} \phi \dif \boldnu^i \to \int_{\Omega \times E} \phi \dif \boldsym{\nu}^\infty$ for any measurable function $\phi : \Omega \times E \to \RR$ such that $\phi(\omega, \cdot)$ is Lipschitz continuous for each $\omega \in \Omega$ and for some $\eta \in \LL^1 (\Omega)$ and $q_0 \in S$,
    \begin{equation*}
    \abs{\phi(\omega,q)} \leq K(\eta (\omega) + d(q_0, q)) \quad
    \omega \in \Omega, q \in S.
    \end{equation*}
\end{enumerate}
\end{proposition}


\begin{proposition}(\cite[Proposition 4.1.2]{castaing2004young})\label{prop:cpt-ible-Young-meas}
Let $\frakK \subset \calY_1(E)$. The following is equivalent:
\begin{enumerate}[(i)]
    \item the set $\frakK$ is relatively compact in $\calS_1(E)$,
    \item the set $\frakK$ is relatively compact in $\calS(E)$ and uniformly integrable in the sense that
    \begin{equation*}
    \lim_{K \to \infty} \sup_{\boldsym{\nu} \in \frakK} \int_{\Omega \times E} d(q_0,q) \mathbf{1}_{\{d(q_0,q) > K\}} \dif \boldsym{\nu}(\omega,q) = 0.
    \end{equation*} 
\end{enumerate}
\end{proposition}

\subsection{Embedding the space of integrable Young measures} \label{appx:embed-Young-meas}

To apply Schauder's fixed-point theorem (\cite[Corollary 17.56]{guide2006infinite}), we need to embed the set $\calP_1 \times \calY_1$ into a locally convex Hausdorff topological vector space.\footnote{This subtlety has already been highlighted in Carmona and Lacker \cite{carmona_lacker_2015} in the proof of their Theorem 3.5.} 
We address the embedding of the set $\calY_1([0,T] \times E)$ of integrable measures equipped with the $\calS_1(E)$-topology for a general Polish space $E$. Our arguments below resemble \cite[p.507]{guide2006infinite}. Similar arguments apply to $\calP_1$. Let $\calF$ be the collection of functions 
\[
    (t,q) \mapsto \mathbf{1}_A(t)\phi(q) 
\]    
for Borel measurable sets $A \subset [0,T]$ and Lipschitz continuous functions $\phi:E \to \RR$.
The product space $\RR^\calF$ is the collection of all mappings from $\calF$ to $\RR$, endowed with the product topology, i.e.~the topology of point-wise convergence on $\calF$. It is well known that $\RR^\calF$ is a locally convex Hausdorff space (c.f.\ Lemma 5.74 in \cite{guide2006infinite}).
We now define a linear map $\iota : \calY_1 \to \RR^\calF$ by
\[
\iota (\boldnu)(\bar{\phi}) := \int_0^T \int_{E} \bar{\phi} \dif \boldnu \quad
\text{for each}~\boldnu \in \calY_1~\text{and}~\bar{\phi} \in \calF.
\]
If we can prove that the above mapping is injective, then the mapping $\iota:\calY_1 \to \iota(\calY_1)$ is homeomorphism 
which shows that $\calY_1$ can be embedded into a locally convex Hausdorff space. 

To prove that $\iota$ is injective we recall the following result on the characterization of weak convergence of probability measures.

\begin{lemma}(\cite[p.15]{castaing2004young})\label{lem:test}
    Let $E$ be a Polish space and let $\operatorname{Lip}(E)$ be the class of real-valued Lipschitz continuous functions on $E$. There exists a countable family $\{\phi^n\}_{n \in \NN} \subset \operatorname{Lip}(S)$ of test functions that determines the weak topology on $\calP(E)$, that is $\nu^i \to \nu$ weakly in $\calP(S)$ if and only if $\int_S \phi^n \dif \nu^i \to \int_S \phi^n \dif \nu$ for all $n \in \NN$. 
\end{lemma}

Let us take any two Young measures $\boldnu, \bar{\boldnu} \in \calY_1$ with respective disintegrations $\nu_\cdot, \bar{\nu}_\cdot$ such that $\iota(\boldnu) = \iota(\bar{\boldnu})$. Then, for any measurable set $A \subset [0,T]$ and any $\phi \in \operatorname{Lip}(E)$ we have that
\[
    \int_A \int_{E} \phi(q) (\nu_t - \bar{\nu}_t)(\dif q) \dif t = 0.
\]
Thus, for each $\phi \in \operatorname{Lip}(E)$ there exists a null set $N_\phi \subset [0,T]$ such that
\begin{equation}\label{null-set-phi}
    \int_{E} \phi \dif \nu_t = \int_{E} \phi \dif \bar{\nu}_t \quad
    \text{for all}~t \in N_\phi^c.        
\end{equation}
In view of Lemma \ref{lem:test}, this shows that there exists a null set $N \subset [0,T]$ such that 
\[
    \nu_t = \bar \nu_t \quad \mbox{for all} \quad t \notin N.
\]
Since the disinteration uniquely determines the Young measures (cf.~Proposition~\ref{prop:disintegration}) this shows that 
$
    \boldnu = \bar{\boldnu}
$
and hence that $\iota$ is indeed injective.

\printbibliography
\end{document}